\documentclass[11pt,twoside]{article}
\usepackage[margin=1.3in]{geometry}

\usepackage[T1]{fontenc}      %

\usepackage{amsfonts,amsmath,amssymb,amsthm}
\usepackage{cases,color,curves}

\usepackage{cite}
\usepackage{etoolbox}

\usepackage{bm,dsfont,extarrows,enumerate,graphicx,ifthen,latexsym,
            listings,makeidx,mathrsfs,psfrag,times,xcolor}

\usepackage[colorlinks,citecolor=green,linkcolor=red]{hyperref}
\usepackage[title]{appendix}
\patchcmd{\thebibliography}{\section*{\refname}}{\section{\refname}\label{sec:references}}{}{}

\def\MR#1{}   %

\makeindex

\newcommand{\R}{\mathbb R}

\newtheorem{theorem}{Theorem}[section]
\newtheorem{lemma}[theorem]{Lemma}

\newtheorem{remark}{Remark}[section]

\newcommand{\dist}{\mathrm{dist}}

\begin{document}
\raggedbottom

\title{\textbf{The direct moving plane method for weak solutions of the fractional $p$-Laplacian}}

\author{ Meiqing Xu \and Hui Yang }

\date{\today}
\maketitle

\begin{abstract}
In this paper, we develop the method of moving planes entirely in the weak formulation for the fractional $p$-Laplacian in the singular range $1<p\leq2$. We first establish a small region principle for antisymmetric functions and apply it to prove radial symmetry and monotonicity of nonnegative weak solutions of fractional $p$-Laplacian equations in a bounded domain. We also consider the nonlocal quasilinear Lane--Emden equation $ (-\Delta)^s_pu=u^q$ in $\mathbb R^n$. In the Sobolev critical case, we establish radial symmetry, monotonicity, and precise asymptotic behavior at infinity for finite-energy weak solutions. Under a suitable decay condition, we also obtain radial symmetry for the full range $q>p-1$. Our proofs rely on integral estimates derived directly from the weak formulation, thereby extending the moving plane approach of Chen and Li (Adv. Math. 2018: 735-758) from the pointwise $C^{1,1}_{\mathrm{loc}}$ setting to weak solutions.

\medskip

\noindent\textit{Keywords:}
fractional $p$-Laplacian; method of moving planes; weak solutions;
small region principle; radial symmetry.

\medskip

\noindent\textit{MSC (2020):}
Primary 35R11, 35B06; Secondary 35B50, 35B51, 35J92.
\end{abstract}

\section{Introduction}\label{sec:introduction}

Let $n\ge2$, $0<s<1$ and $p>1$. Up to a normalization constant, the fractional $p$-Laplacian is defined by
\[
(-\Delta)^s_pu(x)
:=\operatorname{P.V.}\int_{\mathbb R^n}
\frac{|u(x)-u(y)|^{p-2}\bigl(u(x)-u(y)\bigr)}{|x-y|^{n+sp}}\, dy.
\]
This operator arises naturally in the study of various physical phenomena, including the power-law constitutive structure associated with shear-dependent non-Newtonian fluids \cite{MalekRajagopalRuzicka1995}, a long-range interaction kernel related to L\'evy flights and anomalous transport \cite{BG,MJ}.
Evolution problems driven by $(-\Delta)^s_p$  provide natural nonlinear nonlocal diffusion models; see, for instance, \cite{Vazquez2020Evolution,Vazquez2021Sublinear}. The fractional $p$-Laplacian also appears in the Euler-Lagrange equation associated with the fractional Sobolev inequality.
Let $p_s^*:=\frac{np}{n-sp}$. Then the best constant in the fractional Sobolev inequality is determined by
\[
\mathcal S_{s,p}
:=\inf_{0\ne v\in C_c^\infty(\mathbb R^n)}
\frac{\displaystyle
\iint_{\mathbb R^n\times\mathbb R^n}
\frac{|v(x)-v(y)|^p}{|x-y|^{n+sp}}\,dx\,dy}
{\displaystyle
\left(\int_{\mathbb R^n}|v|^{p_s^*}\,dx\right)^{p/p_s^*}}.
\]
After a normalization, its extremals satisfy the critical Euler-Lagrange equation
\[
(-\Delta)^s_pu=u^{p_s^*-1} \quad\text{in }\mathbb R^n.
\]
For the fractional Sobolev spaces and the qualitative properties of extremals, we refer to \cite{MR2944369, BrascoMosconiSquassina2016}.

When $p=2$, the fractional Laplacian operator has been extensively studied in the last two decades. It is well-known that the moving plane method is one of the important tools for studying symmetry of solutions to elliptic equations. However, due to the nonlocal nature of the fractional Laplacian, it is difficult to apply the moving plane method directly. Caffarelli and Silvestre \cite{CaSi}  introduced the extension method, which reduces the nonlocal problem to a local problem in one higher dimension. In this way, the moving plane method and its variant, the moving sphere method, can be applied to study symmetry and classification of fractional order equations; see, e.g., \cite{CZ,FW1,JLX} and the references therein.
The second approach to studying fractional equations is to transform them into integral equations and then apply the moving plane method in integral form introduced by Chen-Li-Ou \cite{CLO3} or the corresponding moving spheres method developed by Li \cite{Ly}. This integral equation method has been successfully used to study fractional problems; see \cite{CLO3,CDQ,Y,ZCCY,DLQSIAM} and the references therein. The third approach is the so-called direct method, which applies the moving plane method directly to the fractional Laplacian equation, based on the maximum principle for antisymmetric functions. Regarding this method, we refer to the earlier works of Chen-Li-Li \cite{CLL}, Jarohs-Weth \cite{JW} and Jin-Xiong \cite{JX}. Later, Chen-Li-Zhang \cite{CLZ} developed the corresponding moving sphere method. Recently, Li-Xu-Yang-Zhuo \cite{li2025direct} refined the direct moving sphere method, removing its dependence on an equivalent integral representation while also weakening the assumptions on the nonlinearity.

For the general fractional $p$-Laplacian ($p \neq 2$), to the best of our knowledge, neither the extension method nor the integral equation method is available. In the framework where solutions belong to $C^{1,1}_{\mathrm{loc}}$, Chen-Li \cite{CL2} established a maximum principle and a boundary point estimate for the fractional $p$-Laplacian, and further developed the direct moving plane method to study the symmetry of positive solutions to the corresponding equation. Subsequently, in the $C^{1,1}_{\mathrm{loc}}$ framework, Wu-Chen \cite{WuChen2020} developed the sliding method for the fractional $p$-Laplacian. For more applications of the direct method, see \cite{WuYuZhang2021,DaiLiuWang2022} and the references therein. On the other hand, there has been substantial progress in the Harnack inequality and H\"{o}lder regularity for the fractional $p$-Laplacian; see \cite{BogeleinEtAlJFA2025,BogeleinEtAlCalcVar2025,DiCastroKuusiPalatucci2014,DiCastroKuusiPalatucci2016,Cozzi2017,IMS2016,MR4788673,BrascoLindgrenSchikorra2018,GarainLindgren2024} and the references therein. Recently, Biswas and Topp \cite{BiswasTopp2025} studied the Lipschitz regularity of the fractional $p$-Laplacian. When $2\le p<2/(1-s)$, Giovagnoli, Jesus and Silvestre \cite{GiovagnoliJesusSilvestre2025} established interior $C^{1,\alpha}$ estimates for the fractional $p$-harmonic functions with some $\alpha \in (0, 1)$.

Therefore, from the perspective of regularity for the fractional $p$-Laplacian, the $C^{1,1}$ assumption in the direct moving plane method of the aforementioned works is too restrictive. More recently, under the assumption $p>2$, Biswas, Roy and Sen \cite{BiswasRoySen2026} developed a moving plane method for $C^1$ positive solutions of the fractional $p$-Laplacian equation on bounded domains. The primary goal of this paper is to develop the moving plane method for weak solutions of the fractional $p$-Laplacian equation in the singular regime $1 < p < 2$, and thereby establish symmetry of positive solutions. For an open set $U\subset\mathbb R^n$, we define
\[
[u]_{W^{s,p}(U)}
:=\left(\int_U\int_U
\frac{|u(x)-u(y)|^p}{|x-y|^{n+sp}}\,dx\,dy\right)^{1/p}
\]
and
\[
W^{s,p}(\mathbb R^n)
:=\left\{u\in L^p(\mathbb R^n):[u]_{W^{s,p}(\mathbb R^n)}<\infty\right\}.
\]
For a bounded domain $\Omega$, let
\[
W^{s,p}_0(\Omega)
:=\left\{u\in W^{s,p}(\mathbb R^n):u=0\ \text{a.e. in }\Omega^c\right\}.
\]
Throughout the paper, $B_r:=B_r(0)\subset\mathbb R^n$. The basic ingredient of our moving plane method in the weak formulation is a small region principle. Without loss of generality, we formulate the bounded-domain results in the unit ball $B_1$. For $-1<\lambda<1$ and $x=(x_1,x_2,\ldots,x_n)$, set
\[
x^\lambda:=(2\lambda-x_1,x_2,\ldots,x_n),\qquad
u_\lambda(x):=u(x^\lambda),
\]
and
\[
\Sigma_\lambda:=\{x\in\mathbb R^n:x_1<\lambda\},\qquad
\Omega_\lambda:=B_1\cap\Sigma_\lambda.
\]

\begin{theorem}[Small region principle for antisymmetric functions]
\label{lem:small-region-p-less-2}
Let $n\ge2$, $0<s<1$ and $1<p\le2$. Let $D$ be an open set such that $D\subset\Omega_\lambda$ with $|D|>0$ and $u\in W^{s,p}_0(B_1)\cap L^\infty(B_1)$ satisfy
\begin{equation}\label{eq:lemma6-assumption}
\begin{cases}
(-\Delta)^s_pu_\lambda-(-\Delta)^s_pu
+c(x)(u_\lambda-u)\ge0
        &\text{weakly in }D,\\
u_\lambda-u\ge0
        &\text{a.e. in }\Sigma_\lambda\setminus D.
\end{cases}
\end{equation}
Suppose that $-1<\lambda<0$ and
\begin{equation}\label{eq:lemma6-c-lower-bound}
        c(x)\ge-M
        \qquad\text{for a.e. }x \in D
\end{equation}
for some $M\ge0$.
Then there exists $\delta=\delta\bigl(n,s,p,M,\|u\|_{L^\infty(B_1)}\bigr)>0$ such that, if $|D|\le\delta$, then
\[
u_\lambda\ge u \qquad\text{a.e. in }\Sigma_\lambda.
\]
\end{theorem}

The first inequality in \eqref{eq:lemma6-assumption}
is understood in the following sense: for every nonnegative
$\zeta\in W_0^{s,p}(D)$
\[
\left\langle
(-\Delta)^s_pu_\lambda-(-\Delta)^s_pu,\zeta
\right\rangle
+\int_D c(x)(u_\lambda-u)\zeta\,dx
\ge0.
\]
The smallness condition in Theorem~\ref{lem:small-region-p-less-2} is imposed on the measure of a set, rather than on the width of a narrow region. Related small-domain and narrow-region principles for local $p$-Laplace equations were developed in \cite{Damascelli1998,DamascelliSciunzi2004,Sciunzi2005,FarinaMontoroSciunzi2012,FarinaMontoroSciunzi2013}. Our proof is entirely variational and requires neither a pointwise minimum nor pointwise evaluation of the operator.

As a first application, we obtain symmetry for a bounded-domain equation with a general locally Lipschitz nonlinearity.

\begin{theorem}\label{Liouville thm}
Let $n\ge2$, $0<s<1$ and $1<p\le2$. Let $f:\mathbb R\to\mathbb R$ be locally Lipschitz continuous with $f(0)\ge0$. Let $u\in W^{s,p}_0(B_1)\cap L^\infty(B_1)$ be a nontrivial nonnegative weak solution of
\begin{equation}\label{main eq}
\begin{cases}
(-\Delta)^s_pu=f(u) &\text{in }B_1,\\
u=0 &\text{in }B_1^c.
\end{cases}
\end{equation}
Then $u$ is radially symmetric with respect to the origin and radially nonincreasing in $B_1$.
\end{theorem}

\begin{remark}
After we submitted the first version of this paper to arXiv, Weimin Zhang
informed the second author that he and Prof. Dong Ye had submitted a preprint on the symmetry of weak solutions in bounded domains that same day.
For \(0<s<\tfrac12\) and \(\tfrac{1}{1-s}<p<2\), Ye and Zhang
\cite[Theorem~1.5]{YeZhang} proved that nonnegative continuous weak
solutions are symmetric about a hyperplane whenever the bounded domain
is symmetric about that hyperplane and convex in its normal direction.
They continue the moving plane by invoking a strong comparison principle
established in this parameter range. In the proof of
Theorem~\ref{Liouville thm}, we use strict separation on a single
interior ball to estimate \((u-u_\lambda)^+\) and thereby continue
the moving plane for all \(0<s<1\) and \(1<p\le 2\).
\end{remark}

The assumption $f(0)\ge0$ is used only to ensure strict positivity of a nontrivial nonnegative solution and is weaker than assuming $f>0$ on $(0,\infty)$. In the special case $f(t)=t^q$, the boundedness assumption can be dropped for $1\le q\le p_s^*-1$, since weak solutions in $W^{s,p}_0(B_1)$ are known to be bounded (see Iannizzotto-Mosconi \cite[Proposition~2.3]{MR4985486}).

We next turn to the equation in the whole space. Define
\[
\mathcal D^{s,p}(\mathbb R^n) := \left\{ v\in L^{p_s^*}(\mathbb R^n): [v]_{W^{s,p}(\mathbb R^n)}<\infty \right\}.
\]

\begin{theorem}\label{thm:critical}
Let $n\ge2$, $0<s<1$ and $1<p\le2$. Let $u\in\mathcal D^{s,p}(\mathbb R^n)\cap C(\mathbb R^n)$
be a nontrivial nonnegative weak solution of
\begin{equation}\label{eq:whole-space-critical-equation}
        (-\Delta)^s_pu=u^{p_s^*-1}
        \quad\text{in }\mathbb R^n.
\end{equation}
Then $u$ is radially symmetric and radially nonincreasing about some point $x_0\in \mathbb R^n$. Moreover, there exists a constant \(A_\infty>0\) such that
\begin{equation}\label{eq:asy}
  u(x) = A_\infty |x-x_0|^{-\frac{n-sp}{p-1}} \bigl(1+o(1)\bigr) \quad\text{as }|x|\to\infty.
\end{equation}
\end{theorem}

Since $L^{p_s^*}(\mathbb R^n)\subset L_{sp}^{p-1}(\mathbb R^n)$ by H\"older's inequality, no separate tail assumption is needed in Theorem~\ref{thm:critical}. In the $C^{1,1}$ framework, Chen-Li \cite{CL2} obtained whole-space symmetry under additional asymptotic hypotheses, while Brasco-Mosconi-Squassina \cite{BrascoMosconiSquassina2016} proved the sharp decay of fractional Sobolev extremals and, more generally, of fixed-sign radial monotone solutions. Theorem~\ref{thm:critical} establishes radial symmetry and monotonicity for any nonnegative weak solution in $L^{p_s^*}(\mathbb R^n)$, and then yields the precise asymptotic behavior with no prior decay assumption.

Finally, we treat the nonlocal quasilinear Lane--Emden equation in the full range $q>p-1$. Define
\[
W^{s,p}_{\mathrm{loc}}(\mathbb R^n)
:=\left\{u\in L^p_{\mathrm{loc}}(\mathbb R^n):
[u]_{W^{s,p}(U)}<\infty\ \text{for every open }U\Subset\mathbb R^n\right\}
\]
and
\[
L_{sp}^{p-1}(\mathbb R^n)
:=\left\{u\in L^{p-1}_{\mathrm{loc}}(\mathbb R^n):
\int_{\mathbb R^n}\frac{|u(x)|^{p-1}}{(1+|x|)^{n+sp}}\,dx<\infty\right\}.
\]
Given $e\in\mathbb S^{n-1}$ and $\lambda\in\mathbb R$, define
\[
x^{\lambda,e}:=x+2(\lambda-x\cdot e)e,
\qquad
u_{\lambda,e}(x):=u(x^{\lambda,e}),
\]
\[
\Sigma_{\lambda,e}:=\{x\in\mathbb R^n:x\cdot e<\lambda\},
\qquad
\Sigma^-_{\lambda,e}:=\{x\in\Sigma_{\lambda,e}:u(x)>u_{\lambda,e}(x)\}.
\]
Let $c_p$ be the constant in \eqref{eq:whole-space-reflected-energy}, and let $c_H$ and $c_H'$ be the constants in Lemma~\ref{lem:whole-space-K-lower}. Fix $c_*=c_*(n,s,p,q)>0$ such that
\begin{equation}\label{c}
2\max\{1,q\}c_*
\le c_p\min\{c_H,c_H'\}.
\end{equation}

\phantomsection
\label{eq:RD-prime}
\noindent\textup{(D)}
We say that $u$ satisfies the \emph{decay condition} if, for every $e\in\mathbb S^{n-1}$, there exists $R_0>0$ such that
\[
\operatorname*{sup}_{\substack{x\in\Sigma^-_{\lambda,e}\\ |x|\ge R}}
|x|^{sp}u(x)^{q-p+1}\le c_*
\qquad\text{for every }R>R_0,
\]
uniformly with respect to $\lambda\in\mathbb R$.

\begin{theorem}\label{thm:general}
Let $n\ge2$, $0<s<1$, $1<p\le2$ and $q>p-1$. Let $u\in W^{s,p}_{\mathrm{loc}}(\mathbb R^n) \cap L_{sp}^{p-1}(\mathbb R^n)\cap C(\mathbb R^n)$
be a nontrivial nonnegative weak solution of
\begin{equation}\label{eq:whole-space-lane-emden}
        (-\Delta)^s_pu=u^q
        \qquad\text{in }\mathbb R^n.
\end{equation}
Suppose that $u$ satisfies the decay condition \hyperref[eq:RD-prime]{\textup{(D)}}. Then $u$ is radially symmetric and radially nonincreasing about some point in $\mathbb R^n$.
\end{theorem}

Theorem~\ref{thm:general} covers the subcritical, critical, and supercritical regimes in a unified way. By contrast, previous moving plane results for the fractional $p$-Laplacian were obtained in the pointwise framework and required both higher regularity and global asymptotic control \cite{CL2,WuChen2020,WuYuZhang2021,DaiLiuWang2022}. Our condition \hyperref[eq:RD-prime]{\textup{(D)}} is placed only on the sets where the desired reflection inequality fails; it is satisfied, for instance, whenever \[
|x|^{sp}u(x)^{q-p+1}\to0
\quad\text{as }|x|\to\infty.
\]
In the subcritical range $p-1<q<p_s^*-1$, the continuity assumption in Theorem~\ref{thm:general} is actually redundant: Lemma~\ref{lem:subcritical-regularity} shows that every nonnegative weak solution in $W^{s,p}_{\mathrm{loc}}(\mathbb R^n)\cap L_{sp}^{p-1}(\mathbb R^n)$ admits a locally H\"older continuous representative. Below the Serrin exponent, the Liouville theorems for weak solutions were obtained by Liu \cite{liu2025nontrivial}.

The moving plane method was first introduced by Alexandrov and later used by Serrin \cite{Serrin1971} to prove symmetry in elliptic PDEs. Gidas--Ni--Nirenberg \cite{GNN} then established the classical symmetry result for positive solutions of semilinear elliptic equations. This method became central to the classification theorems of Caffarelli--Gidas--Spruck \cite{CGS} and Chen--Li \cite{CL1} for the Lane--Emden equation.
For the moving plane method applied to local $p$-Laplacian equations, we refer to \cite{Damascelli1998,DamascelliPacella1998,DamascelliPacellaRamaswamy1999,DamascelliSciunzi2004,FarinaMontoroSciunzi2012,DamascelliMerchanMontoroSciunzi2014,Sciunzi2016,OlivaSciunziVaira2020,SerrinZou1999} and the references therein.

The paper is organized as follows. In Section~\ref{pre}, we introduce the weak formulation, justify the antisymmetric truncations as admissible test functions, and derive the basic reflection identities and algebraic estimates. Section~\ref{sec:smal} proves the small region principle in Theorem \ref{lem:small-region-p-less-2}. Section~\ref{sec:bounded-domain-proof} proves Theorem~\ref{Liouville thm}. Section~\ref{sec:whole-space-RD-proof} proves Theorem~\ref{thm:general}. Section~\ref{sec:critical-energy-proof} proves Theorem~\ref{thm:critical}.

\section{Preliminaries}\label{pre}

For every admissible test function \(\varphi\), we define
\[
\begin{aligned}
\left\langle(-\Delta)^s_pu,\varphi\right\rangle
&:=\frac12
\iint_{\mathbb R^n\times\mathbb R^n}
\frac{|u(x)-u(y)|^{p-2}\bigl(u(x)-u(y)\bigr)\bigl(\varphi(x)-\varphi(y)\bigr)}
{|x-y|^{n+sp}}\,dx\,dy\\
&=\frac12\iint_{\mathbb R^n\times\mathbb R^n}
\frac{G\bigl(u(x)-u(y)\bigr)\bigl(\varphi(x)-\varphi(y)\bigr)}
{|x-y|^{n+sp}}\,dx\,dy,
\end{aligned}
\]
where $G(t):=|t|^{p-2}t$. Writing \(p_s^*:=\frac{np}{n-sp}\), a function \(u\in W^{s,p}_0(B_1)\)
is a \textit{weak solution} of \eqref{main eq} if
\[
\langle (-\Delta)^s_p u,\varphi\rangle = \int_{B_1} f(u)\varphi\,dx \qquad \text{for every } \varphi\in W^{s,p}_0(B_1).
\]
A nonnegative function $u\in W^{s,p}_{\mathrm{loc}}(\mathbb R^n)
\cap L_{sp}^{p-1}(\mathbb R^n)
\cap L^q_{\mathrm{loc}}(\mathbb R^n)$ is a \textit{weak solution} of \eqref{eq:whole-space-lane-emden} if
\begin{equation}\label{whold weak}
   \left\langle(-\Delta)^s_pu,\zeta\right\rangle
 =\int_{\mathbb R^n}u^q\zeta\,dx
 \qquad\text{for every }\zeta\in C_c^\infty(\mathbb R^n).
\end{equation}
By a standard density argument, the weak formulation extends to
compactly supported test functions in $W^{s,p}(\mathbb R^n)$ whenever
the right-hand side is finite.
For later use, denote the reflecting hyperplane by
\[
T_{\lambda,e}:=\{x\in\mathbb R^n:x\cdot e=\lambda\}.
\]
When \(e=e_1\), we continue to suppress the directional subscript and write
\[
T_\lambda:=\{x\in\mathbb R^n:x_1=\lambda\}, \qquad \Sigma_\lambda^-:=\{x\in\Sigma_\lambda:u>u_\lambda\}.
\]
For \(-1<\lambda<1\), set
\begin{equation}\label{def:Omega}
 \Omega_\lambda^-:=\{x\in\Omega_\lambda:u>u_\lambda\}.
\end{equation}
Set
\begin{equation}\label{eq def of test func}
    \psi:=u-u_\lambda,
        \qquad
        \varphi:=\psi^+\chi_{\Sigma_\lambda}.
\end{equation}
For \(\varepsilon>0\), define
\begin{equation}\label{eq:whole-space-shifted-truncation}
 \varphi_\varepsilon:=
 (\psi-\varepsilon)^+\chi_{\Sigma_\lambda},
\end{equation}
and
\begin{equation}\label{eq:A}
   A_{\lambda,\varepsilon}
 :=\{x\in\Sigma_\lambda:
 \psi(x)>\varepsilon\}.
\end{equation}

The following two lemmas verify that \(\varphi\) and \(\varphi_\varepsilon\) are admissible test functions for the weak equations of \eqref{main eq} and \eqref{eq:whole-space-lane-emden}, respectively.

\begin{lemma}
\label{lem:bounded-domain-test-admissibility}
Let \(n\ge 2\), \(0<s<1\) and \(1<p\le2\). Let \(u\in W_0^{s,p}(B_1)\), let \(-1<\lambda<1\), and let \(\psi\) and \(\varphi\) be defined by \eqref{eq def of test func}. Then
\[
\psi,\ \psi^+\in W^{s,p}(\mathbb R^n), \qquad \varphi\in W^{s,p}(\R^n).
\]
\end{lemma}

\begin{proof}
One can see that \(u_\lambda\in W^{s,p}(\mathbb R^n)\), and hence \(\psi=u-u_\lambda\in W^{s,p}(\mathbb R^n)\). Since
\[
|\psi^+(x)-\psi^+(y)| \le |\psi(x)-\psi(y)|,
\]
we also have \(\psi^+\in W^{s,p}(\mathbb R^n)\). For \(x,y\in\Sigma_\lambda\), we have
\[
\frac{1}{|x-y|^{n+sp}} \ge \frac{1}{|x-y^\lambda|^{n+sp}}>0.
\]
For every \(a,b\in\mathbb R\) and \(A\ge B>0\),
\[
A|a^+-b^+|^p+B\bigl((a^+)^p+(b^+)^p\bigr) \le A|a-b|^p+B|a+b|^p.
\]
Using this inequality and the antisymmetry \(\psi(y^\lambda)=-\psi(y)\), we obtain
\[
\begin{aligned}
 \bigl[\varphi\bigr]_{W^{s,p}(\mathbb R^n)}^p
 &=\iint_{\Sigma_\lambda\times\Sigma_\lambda}
 \left\{
 \frac{|\psi^+(x)-\psi^+(y)|^p}{|x-y|^{n+sp}}
 +\frac{(\psi^+(x))^p+(\psi^+(y))^p}
 {|x-y^\lambda|^{n+sp}}
 \right\}\,dx\,dy\\
 &\le\iint_{\Sigma_\lambda\times\Sigma_\lambda}
 \left\{
 \frac{|\psi(x)-\psi(y)|^p}{|x-y|^{n+sp}}
 +\frac{|\psi(x)+\psi(y)|^p}{|x-y^\lambda|^{n+sp}}
 \right\}\,dx\,dy\\
 &=\frac12\bigl[\psi\bigr]_{W^{s,p}(\mathbb R^n)}^p.
\end{aligned}
\]
Since \(|\varphi|\le|\psi|\), we also have \(\varphi\in L^p(\mathbb R^n)\). Thus \(\varphi\in W^{s,p}(\mathbb R^n)\).
\end{proof}

Next we verify that \(\varphi_\varepsilon\) is an admissible test function for the weak equations of \eqref{eq:whole-space-lane-emden}.
\begin{lemma}
\label{lem:whole-space-shifted-admissibility}
Let \(n\ge 2\), \(0<s<1\), \(1<p\le2\) and $q>p-1$.  Let $0\le u\in W^{s,p}_{\mathrm{loc}}(\mathbb R^n) \cap L_{sp}^{p-1}(\mathbb R^n) \cap C(\mathbb R^n)$ satisfy the decay condition \hyperref[eq:RD-prime]{\textup{(D)}}.
Let $\varphi_\varepsilon$ be defined by \eqref{eq:whole-space-shifted-truncation}. Then $\varphi_\varepsilon \in W^{s,p}(\mathbb R^n)$ has compact support and is a valid test function in \eqref{whold weak} for both \(u\) and \(u_\lambda\).
\end{lemma}

\begin{proof}
Since $0<\psi\le u$ on \(\Sigma_\lambda^-\), the decay condition \hyperref[eq:RD-prime]{\textup{(D)}} implies that \(\psi(x)\to0\) uniformly on \(\Sigma_\lambda^-\) as \(|x|\to\infty\). Hence for every \(\lambda\in\mathbb R\) and \(\varepsilon>0\), \(A_{\lambda,\varepsilon}\) defined by \eqref{eq:A} is bounded. Since \(\psi=0\) on \(T_\lambda\) and \(\psi\) is continuous, \(A_{\lambda,\varepsilon}\) has positive distance from \(T_\lambda\).
Choose an open set \(U\) and \(\chi\in C_c^\infty(\Sigma_\lambda)\) such that
\[
\overline{A_{\lambda,\varepsilon}} \subset U\Subset\Sigma_\lambda, \qquad 0\le\chi\le1, \qquad \chi\equiv1\quad\text{in }U.
\]
Then $\varphi_\varepsilon=\chi(\psi-\varepsilon)^+$.
Since \(u\in W_{\mathrm{loc}}^{s,p}(\mathbb R^n)\), one has $u_\lambda,\psi \in W_{\mathrm{loc}}^{s,p}(\mathbb R^n)$. Since
\[
\left|(\psi(x)-\varepsilon)^+ -(\psi(y)-\varepsilon)^+\right| \le |\psi(x)-\psi(y)|,
\]
we have \((\psi-\varepsilon)^+\in W_{\mathrm{loc}}^{s,p}(\mathbb R^n)\). Choose \(R\ge1\) such that \(\operatorname{supp}\chi\subset B_R\). For \(x,y\in B_{2R}\),
\[
\begin{aligned}
|\varphi_\varepsilon(x)-\varphi_\varepsilon(y)|
&\le
\|\chi\|_\infty
\left|(\psi(x)-\varepsilon)^+
-(\psi(y)-\varepsilon)^+\right|+
\|\nabla\chi\|_\infty
(\psi(y)-\varepsilon)^+|x-y|.
\end{aligned}
\]
Consequently,
\[
\begin{aligned}
&\iint_{B_{2R}\times B_{2R}}
\frac{|\varphi_\varepsilon(x)-\varphi_\varepsilon(y)|^p}
{|x-y|^{n+sp}}\,dx\,dy\\
&\quad\le
C[(\psi-\varepsilon)^+]_{W^{s,p}(B_{2R})}^p
+C\int_{B_{2R}}|(\psi(y)-\varepsilon)^+|^p
\int_{B_{2R}}|x-y|^{-n+p-sp}\,dx\,dy\\
&\quad\le
C[(\psi-\varepsilon)^+]_{W^{s,p}(B_{2R})}^p
+CR^{p-sp}\|(\psi-\varepsilon)^+\|_{L^p(B_{2R})}^p
<\infty.
\end{aligned}
\]
Since \(\varphi_\varepsilon=\chi(\psi-\varepsilon)^+\) is supported in \(B_R\),
\[
\begin{aligned}
&2\int_{B_R}\int_{B_{2R}^c}
\frac{|\varphi_\varepsilon(x)|^p}
{|x-y|^{n+sp}}\,dy\,dx  \le
CR^{-sp}\|\varphi_\varepsilon\|_{L^p(B_R)}^p<\infty.
\end{aligned}
\]
Therefore $\varphi_\varepsilon \in W^{s,p}(\mathbb R^n)$ and it has compact support. By a standard density argument, $\varphi_\varepsilon$ is a valid test function for \eqref{whold weak} satisfied by both \(u\) and \(u_\lambda\).
\end{proof}

In the subcritical Lane--Emden case, the continuity assumption in Theorem~\ref{thm:general} is not necessary. Indeed, using the estimates in \cite{Cozzi2017}, one can verify that the solution admits a locally H\"older continuous representative.
\begin{lemma}
\label{lem:subcritical-regularity}
Let \(1<p\le 2\) and $0<q<\frac{np}{n-sp}-1$. Suppose that $0\le u\in W_{\mathrm{loc}}^{s,p}(\mathbb R^n) \cap L_{sp}^{p-1}(\mathbb R^n)$ is a weak solution of \eqref{eq:whole-space-lane-emden}. Then there exists \(\alpha\in(0,1)\), depending only on \(n,s,p\), such that \(u\) admits a representative satisfying $u\in C^\alpha_{\mathrm{loc}}(\mathbb R^n)$.
\end{lemma}

\begin{proof}
Fix \(x_0\in\mathbb R^n\) and \(R>0\). Choose \(\eta\in C_c^\infty(B_{4R}(x_0))\) such that
\[
0\le\eta\le1, \qquad \eta\equiv1 \quad\text{in }B_{3R}(x_0),
\]
and set $v:=\eta u$. We first observe that $v\in W^{s,p}(\mathbb R^n)$. Indeed, \(v\in L^p(\mathbb R^n)\) because it has compact support and \(u\in L^p_{\mathrm{loc}}(\mathbb R^n)\). For \(x,y\in B_{5R}(x_0)\), we have
\[
\begin{aligned}
|v(x)-v(y)|
&\le
|\eta(x)|\,|u(x)-u(y)|
+
|u(y)|\,|\eta(x)-\eta(y)| \\
&\le |\eta(x)|\,|u(x)-u(y)|
+
C_R|u(y)|\,|x-y|.
\end{aligned}
\]
It follows that
\begin{align*}
&\int_{B_{5R}(x_0)}
 \int_{B_{5R}(x_0)}
 \frac{|v(x)-v(y)|^p}{|x-y|^{n+sp}}\,dx\,dy
\\
&\qquad\le
C [u]_{W^{s,p}(B_{5R}(x_0))}^p
+
C_R
\int_{B_{5R}(x_0)}
|u(y)|^p
\int_{B_{5R}(x_0)}
\frac{dx}{|x-y|^{n+sp-p}}\,dy
<\infty.
\end{align*}
Moreover, since \(\operatorname{supp}v\subset B_{4R}(x_0)\),
\[
\int_{B_{5R}(x_0)} \int_{\mathbb R^n\setminus B_{5R}(x_0)} \frac{|v(x)|^p}{|x-y|^{n+sp}}\,dy\,dx \le C R^{-sp}\|v\|_{L^p(B_{4R}(x_0))}^p<\infty.
\]
Thus \(v\in W^{s,p}(\mathbb R^n)\).

For \(x\in B_{2R}(x_0)\), define
\[
h_R(x) := \int_{\mathbb R^n\setminus B_{3R}(x_0)} \frac{ G\bigl(u(x)-v(y)\bigr) - G\bigl(u(x)-u(y)\bigr) }{ |x-y|^{n+sp} }\,dy.
\]
Since \(v=u\) in \(B_{3R}(x_0)\), for every \(\zeta\in C_c^\infty(B_{2R}(x_0))\), a direct decomposition of the double integrals gives
\[
\left\langle (-\Delta)^s_pv-(-\Delta)^s_pu,\zeta \right\rangle = \int_{B_{2R}(x_0)}h_R(x)\zeta(x)\,dx.
\]
Consequently,
\[
(-\Delta)^s_pv=v^q+h_R(x) \qquad\text{weakly in }B_{2R}(x_0).
\]
We claim that $h_R\in L^\infty(B_{2R}(x_0))$. Since \(1<p\le2\), the function \(G\) satisfies
\[
|G(a)-G(b)| \le C_p|a-b|^{p-1} \qquad \text{for all }a,b\in\mathbb R.
\]
Therefore, for \(x\in B_{2R}(x_0)\),
\begin{align*}
|h_R(x)|
&\le
C_p
\int_{\mathbb R^n\setminus B_{3R}(x_0)}
\frac{|u(y)-v(y)|^{p-1}}{|x-y|^{n+sp}}\,dy
\le
C_p
\int_{\mathbb R^n\setminus B_{3R}(x_0)}
\frac{|u(y)|^{p-1}}{|x-y|^{n+sp}}\,dy \\
& \le
C
\int_{\mathbb R^n\setminus B_{3R}(x_0)}
\frac{|u(y)|^{p-1}}{|y-x_0|^{n+sp}}\,dy
<\infty,
\end{align*}
where the last inequality follows from \(u\in L_{sp}^{p-1}(\mathbb R^n)\). Hence $\|h_R\|_{L^\infty(B_{2R}(x_0))}<\infty$.

Set $f_R(x,t):=\bigl((t^+)^q+h_R(x)\bigr)$, which satisfies
\[
|f_R(x,t)| \le \|h_R\|_{L^\infty(B_{2R}(x_0))} + |t|^{q}.
\]
Up to a positive constant, $v$ satisfies
\[
(-\Delta)^s_pv=f_R(x,v) \qquad\text{weakly in }B_{2R}(x_0).
\]
Thus the assumptions of \cite[Theorem~8.1]{Cozzi2017} are satisfied. It follows that $v\in L^\infty_{\mathrm{loc}}(B_{2R}(x_0))$. Moreover, \(f_R(x,t)\) is locally bounded in \(t\), uniformly with respect to \(x\in B_{2R}(x_0)\). Hence \cite[Theorem~8.2]{Cozzi2017} yields $v\in C^\alpha_{\mathrm{loc}}(B_{2R}(x_0))$ for some \(\alpha\in(0,1)\) depending only on \(n,s,p\). We conclude that $u\in C^\alpha(B_R(x_0))$.
\end{proof}

The following lemma gives a decomposition of the weak form of $(-\Delta)^s_p u-(-\Delta)^s_p u_\lambda$ and establishes the corresponding sign estimates.
\begin{lemma}\label{lem:reflection-inequality}
Consider $1<p\le 2$ and $-1<\lambda<1$. Let $0\le u\in W_0^{s,p}(B_1)$. Assume that $\varphi$ is defined by \eqref{eq def of test func}. Then
\begin{equation}\label{eq:phi>0}
     \left\langle
        (-\Delta)^s_p u-(-\Delta)^s_p u_\lambda,\varphi
        \right\rangle \ge0.
\end{equation}
Moreover, let $\lambda\in\mathbb{R}$ and  $0\le u\in W^{s,p}_{\mathrm{loc}}(\mathbb R^n) \cap L_{sp}^{p-1}(\mathbb R^n) \cap C(\mathbb R^n)$ satisfy the decay condition \hyperref[eq:RD-prime]{\textup{(D)}}, and suppose that $\varphi_\varepsilon$ is defined by \eqref{eq:whole-space-shifted-truncation}. Then
\begin{equation}\label{eq:phi-epsilon-positive}
\left\langle
(-\Delta)_{p}^{s}u-(-\Delta)_{p}^{s}u_{\lambda},\varphi_\varepsilon
\right\rangle\geq0,
\end{equation}
and
\begin{equation}\label{eq:whole-space-reflected-energy}
  \begin{aligned}
 &\quad \left\langle
 (-\Delta)^s_pu-(-\Delta)^s_pu_\lambda,\varphi_\varepsilon
 \right\rangle\\
  &\ge c_p\int_{A_{\lambda,\varepsilon}}
 \psi(x)\varphi_\varepsilon(x)
 \int_{\Sigma_\lambda}
 \frac{
 \bigl(
 |u(x)-u(y)|+|u_\lambda(x)-u(y)|
 \bigr)^{p-2}}
 {|x-y^\lambda|^{n+sp}}\,dy\,dx.
\end{aligned}
\end{equation}
\end{lemma}

\begin{proof}
Since \(\varphi=0\) a.e. in \(\Omega_\lambda^c\), we have \(\varphi\in W_0^{s,p}(\Omega_\lambda)\). Thus $\varphi$ is an admissible test function. Write
\[
\left\langle (-\Delta)^s_pu-(-\Delta)^s_pu_\lambda, \varphi \right\rangle =I_1+I'_2+I''_2,
\]
where
\begin{equation}\label{eq:I1}
  \begin{aligned}
I_1
&:=\frac12\iint_{\Sigma_\lambda\times\Sigma_\lambda}
\left(
\frac1{|x-y|^{n+sp}}-\frac1{|x-y^\lambda|^{n+sp}}
\right)
\bigl[G(u(x)-u(y))-G(u_\lambda(x)-u_\lambda(y))\bigr]\\
&\hspace{42mm}\cdot
\bigl(\varphi(x)-\varphi(y)\bigr)\,dx\,dy,
\end{aligned}
\end{equation}
\[
\begin{aligned}
I'_2
&:=\iint_{\Sigma_\lambda\times\Sigma_\lambda}
\frac1{|x-y^\lambda|^{n+sp}}
\bigl[G(u(x)-u(y))-G(u_\lambda(x)-u(y))\bigr]
\varphi(x)\,dx\,dy,
\end{aligned}
\]
and
\[
\begin{aligned}
I''_2
&:=\iint_{\Sigma_\lambda\times\Sigma_\lambda}
\frac1{|x-y^\lambda|^{n+sp}}
\bigl[G(u(x)-u_\lambda(y))-G(u_\lambda(x)-u_\lambda(y))\bigr]
\varphi(x)\,dx\,dy.
\end{aligned}
\]
We claim that the integrands of $I_1$, $I_2'$ and $I_2''$ are nonnegative. First, by the elementary identity
\[
G(b)-G(a) =(p-1)(b-a)\int_0^1 |a+t(b-a)|^{p-2}\, dt,
\]
we get
\begin{equation}\label{eq 2.7}
    \begin{aligned}
&G(u(x)-u(y))-G(u_\lambda(x)-u_\lambda(y))                      \\
&\quad=
(p-1)(\psi(x)-\psi(y))
\int_0^1
\left|
 u_\lambda(x)-u_\lambda(y)+t(\psi(x)-\psi(y))
\right|^{p-2}\,dt.
\end{aligned}
\end{equation}
Moreover, for every $a,b\in\R$,
\begin{equation}\label{eq 2.8}
         (a-b)(a^+-b^+)
        \ge |a^+-b^+|^2.
\end{equation}
Moreover, we have
\begin{equation}\label{eq 2.2}
       \frac1{|x-y|^{n+sp}}>\frac1{|x-y^\lambda|^{n+sp}}.
\end{equation}
Using \eqref{eq 2.2}, \eqref{eq 2.7}, and \eqref{eq 2.8}, we obtain that the integrand of $I_1$ is nonnegative. Since $G$ is nondecreasing, for every fixed $x\in\{\varphi> 0\}$ and $y\in\Sigma_\lambda$,  we have
\[
G(u(x)-u(y))-G(u_\lambda(x)-u(y))\ge0,
\]
and
\[
G(u(x)-u_\lambda(y))-G(u_\lambda(x)-u_\lambda(y))\ge0.
\]
Therefore the integrands of $I_2'$ and $I_2''$ are nonnegative. Thus
\[
\left\langle (-\Delta)^s_pu-(-\Delta)^s_pu_\lambda, \varphi \right\rangle =I_1+I'_2+I''_2\ge0.
\]
Similarly, we decompose $ \left\langle (-\Delta)^s_pu-(-\Delta)^s_pu_\lambda, \varphi_\varepsilon \right\rangle$ into $J_1+J'_2+J''_2$, where
\[
\begin{aligned}
J_1
&:=\frac12\iint_{\Sigma_\lambda\times\Sigma_\lambda}
\left(
\frac1{|x-y|^{n+sp}}-\frac1{|x-y^\lambda|^{n+sp}}
\right)
\bigl[G(u(x)-u(y))-G(u_\lambda(x)-u_\lambda(y))\bigr]\\
&\hspace{42mm}\cdot
\bigl(\varphi_\varepsilon(x)-\varphi_\varepsilon(y)\bigr)\,dx\,dy,
\end{aligned}
\]
\[
\begin{aligned}
J'_2
&:=\iint_{\Sigma_\lambda\times\Sigma_\lambda}
\frac1{|x-y^\lambda|^{n+sp}}
\bigl[G(u(x)-u(y))-G(u_\lambda(x)-u(y))\bigr]
\varphi_\varepsilon(x)\,dx\,dy,
\end{aligned}
\]
and
\[
\begin{aligned}
J''_2
&:=\iint_{\Sigma_\lambda\times\Sigma_\lambda}
\frac1{|x-y^\lambda|^{n+sp}}
\bigl[G(u(x)-u_\lambda(y))-G(u_\lambda(x)-u_\lambda(y))\bigr]
\varphi_\varepsilon(x)\,dx\,dy.
\end{aligned}
\]
Since \(t\mapsto(t-\varepsilon)^+\) is nondecreasing, as in the proof of \eqref{eq:phi>0}, the integrands of $J_1$, $J_2'$ and $J_2''$ are nonnegative. Thus \eqref{eq:phi-epsilon-positive} holds.

For \(1<p\le2\), one can verify the inequality
\begin{equation}\label{eq:lemma6-monotonicity-p-less-2}
 G(a)-G(b)
 \ge c_p(|a|+|b|)^{p-2}(a-b)
 \qquad\text{whenever }a>b.
\end{equation}
Hence
\begin{equation*}
  \begin{aligned}
 &\quad \left\langle
 (-\Delta)^s_pu-(-\Delta)^s_pu_\lambda,\varphi_\varepsilon
 \right\rangle\\
& \ge J'_2=\int_{A_{\lambda,\varepsilon}}\int_{\Sigma_\lambda}
 \frac{
 G(u(x)-u(y))-G(u_\lambda(x)-u(y))}
 {|x-y^\lambda|^{n+sp}}
 \varphi_\varepsilon(x)\,dy\,dx\\
  &\ge c_p\int_{A_{\lambda,\varepsilon}}
 \psi(x)\varphi_\varepsilon(x)
 \int_{\Sigma_\lambda}
 \frac{
 \bigl(
 |u(x)-u(y)|+|u_\lambda(x)-u(y)|
 \bigr)^{p-2}}
 {|x-y^\lambda|^{n+sp}}\,dy\,dx.
\end{aligned}
\end{equation*}
The last inequality is due to \eqref{eq:lemma6-monotonicity-p-less-2}.
\end{proof}

We then give a geometric estimate on a half-space, which will be used in the small region principle.
\begin{lemma}
\label{lem:halfspace-exterior-potential}
Let \(0<s<1\), \(1<p<\infty\), and let \(H\subset \mathbb R^n\) be an open half-space. Let \(D\subset H\) be a measurable bounded set, and $0<|D|<\infty$. Then there exists a constant \(C=C(n,s,p)>0\) such that, for every \(x\in H\),
\[
\int_{H\setminus D} \frac{dy}{|x-y|^{n+sp}} \ge C|D|^{-sp/n}.
\]
\end{lemma}

\begin{proof}
Let \(\omega_n:=|B_1(0)|\), and choose $R:=\left(\frac{4|D|}{\omega_n}\right)^{1/n}$. For every \(x\in H\) we have
\[
|B_R(x)\cap H|\ge \frac12 \omega_n R^n = 2|D|.
\]
Therefore
\[
|(B_R(x)\cap H)\setminus D| \ge |B_R(x)\cap H|-|D| \ge |D|.
\]
Consequently,
\[
\begin{aligned}
\int_{H\setminus D}
\frac{dy}{|x-y|^{n+sp}}
&\ge
\int_{(B_R(x)\cap H)\setminus D}
\frac{dy}{|x-y|^{n+sp}} \ge
R^{-n-sp}|(B_R(x)\cap H)\setminus D|  \\
&\ge
|D|R^{-n-sp} \ge
C(n,s,p)|D|^{-sp/n}.
\end{aligned}
\]
\end{proof}

The following algebraic inequality will be used in the proof of Theorem~\ref{Liouville thm}.
\begin{lemma}
\label{lem:two-scale-algebraic-inequality}
Let \(a,b,\alpha>0\). Then there exists \(c=c(a,b,\alpha)>0\) such that, for every \(d>0\) and \(t\ge0\),
\begin{equation*}
        ad^{-\alpha}t^2+bdt
        \ge c t^{(\alpha+2)/(\alpha+1)}.
\end{equation*}
\end{lemma}

\begin{proof}
The assertion is immediate for \(t=0\). If \(t>0\), set $c:=\inf\limits_{r>0}\bigl(ar^{-\alpha}+br\bigr)$ and \(r=dt^{-1/(\alpha+1)}\). Then $c$ is positive, because \(ar^{-\alpha}+br\) tends to \(+\infty\) as \(r\downarrow0\) or \(r\to\infty\). Then
\[
ad^{-\alpha}t^2+bdt =t^{(\alpha+2)/(\alpha+1)} \bigl(ar^{-\alpha}+br\bigr)\ge c t^{(\alpha+2)/(\alpha+1)}.
\]
\end{proof}

We also need the following inequality to estimate the difference.
\begin{lemma}\label{lem:power-difference-all-q}
Let \(q>0\), and set
\[
c_q:=\max\{1,q\}, \qquad \vartheta_q:=\min\{1,q\}.
\]
For every \(0\le b<a\),
\begin{equation}\label{eq:power-difference-linearized}
        0<a^q-b^q
        \le c_q a^{q-1}(a-b).
\end{equation}
Moreover, for every \(M>0\) and \(0\le b<a\le M\),
\begin{equation}\label{eq:power-difference-holder}
        a^q-b^q
        \le L_{q,M}(a-b)^{\vartheta_q},
\end{equation}
where
\[
L_{q,M}:=
        \begin{cases}
        1, & 0<q\le1,\\
        qM^{q-1}, & q>1.
        \end{cases}
\]
\end{lemma}

\begin{proof}
If \(q>1\), both estimates follow from the mean value theorem. Let \(0<q\le1\). Writing \(t=b/a\in[0,1)\), we obtain
\[
a^q-b^q=a^q(1-t^q) \le a^q(1-t)=a^{q-1}(a-b),
\]
because \(t^q\ge t\). This proves \eqref{eq:power-difference-linearized}. Moreover, the subadditivity of \(r\mapsto r^q\) on \([0,\infty)\) gives
\[
a^q=\bigl(b+(a-b)\bigr)^q \le b^q+(a-b)^q,
\]
and hence \eqref{eq:power-difference-holder}.
\end{proof}

\section{Proof of the small region principle}\label{sec:smal}

\begin{proof}[Proof of Theorem~\ref{lem:small-region-p-less-2}]
Let \(\psi\) and \(\varphi\) be given by \eqref{eq def of test func}. The second condition in \eqref{eq:lemma6-assumption} implies that \(\Sigma_\lambda^-\subset D\) up to a null set. Using the decomposition from the proof of Lemma~\ref{lem:reflection-inequality}, write
\[
\left\langle (-\Delta)^s_pu-(-\Delta)^s_pu_\lambda, \varphi \right\rangle =I_1+I'_2+I''_2.
\]
Thus \eqref{eq:lemma6-monotonicity-p-less-2} and \eqref{eq 2.8} imply
\begin{equation}\label{eq:23}
\begin{aligned}
&\bigl[G(u(x)-u(y))-G(u_\lambda(x)-u_\lambda(y))\bigr]
 \bigl(\varphi(x)-\varphi(y)\bigr)\\
\ge&
 c_p(|u(x)-u(y)|+|u_\lambda(x)-u_\lambda(y)|)^{p-2}
 \bigl(\psi(x)-\psi(y)\bigr)
 \bigl(\varphi(x)-\varphi(y)\bigr) \\
\ge &C_p \|u\|_{L^\infty(B_1)}^{p-2}\bigl(\psi(x)-\psi(y)\bigr)
 \bigl(\varphi(x)-\varphi(y)\bigr) \\
\ge &C_p \|u\|_{L^\infty(B_1)}^{p-2}|\varphi(x)-\varphi(y)|^2,
\end{aligned}
\end{equation}
for some $C_p>0$.
Using \eqref{eq:I1}, we have
\begin{equation}\label{eq:reflected-energy-I1}
\begin{aligned}
        I_1
        &\ge C_1
        \iint_{\Sigma_\lambda\times\Sigma_\lambda}
        \left(
        \frac1{|x-y|^{n+sp}}-\frac1{|x-y^\lambda|^{n+sp}}
        \right)
        |\varphi(x)-\varphi(y)|^2\,dx\,dy,
\end{aligned}
\end{equation}
where $C_1=C_1(n,s,p,\|u\|_{L^\infty(B_1)})$. We next estimate $I_2'$ and $I_2''$. \eqref{eq:lemma6-monotonicity-p-less-2} gives
\[
\begin{aligned}
 &\bigl[G(u(x)-u(y))-G(u_\lambda(x)-u(y))\bigr]\varphi(x)\\
 \ge& c_p(|u(x)-u(y)|+|u_\lambda(x)-u(y)|)^{p-2}\psi(x)\varphi(x)\\
 \ge& C_1\varphi(x)^2,
\end{aligned}
\]
after decreasing \(C_1>0\) if necessary. Thus
\[
I'_2\ge C_1\iint_{\Sigma_\lambda\times\Sigma_\lambda} \frac{\varphi(x)^2}{|x-y^\lambda|^{n+sp}}\,dx\,dy.
\]
Similarly we have
\[
I''_2\ge C_1\iint_{\Sigma_\lambda\times\Sigma_\lambda} \frac{\varphi(x)^2}{|x-y^\lambda|^{n+sp}}\,dx\,dy = C_1\iint_{\Sigma_\lambda\times\Sigma_\lambda} \frac{\varphi(y)^2}{|x-y^\lambda|^{n+sp}}\,dx\,dy.
\]
Therefore,
\begin{equation}\label{eq:reflected-energy-I2}
\begin{aligned}
 I'_2+I''_2 \ge   C_1  \iint_{\Sigma_\lambda\times\Sigma_\lambda}
  \frac{\varphi(x)^2+\varphi(y)^2}{|x-y^\lambda|^{n+sp}} \,dx\,dy \ge   C_1  \iint_{\Sigma_\lambda\times\Sigma_\lambda}
  \frac{|\varphi(x)-\varphi(y)|^2}{|x-y^\lambda|^{n+sp}} \,dx\,dy.
\end{aligned}
\end{equation}
\eqref{eq:reflected-energy-I1}--\eqref{eq:reflected-energy-I2} imply
\begin{equation}\label{eq:reflected-energy-full-kernel}
        \left\langle
        (-\Delta)^s_pu-(-\Delta)^s_pu_\lambda,
        \varphi
        \right\rangle
        \ge C_1
        \iint_{\Sigma_\lambda\times\Sigma_\lambda}
        \frac{|\varphi(x)-\varphi(y)|^2}{|x-y|^{n+sp}}
        \,dx\,dy.
\end{equation}
Since \(\varphi=0\) a.e. in \(\Sigma_\lambda\setminus\Sigma_\lambda^-\), Lemma~\ref{lem:halfspace-exterior-potential} and \eqref{eq:reflected-energy-full-kernel} give
\begin{equation}\label{eq:reflected-energy-small-support}
\begin{aligned}
        \left\langle
        (-\Delta)^s_pu-(-\Delta)^s_pu_\lambda,
        \varphi
        \right\rangle
        &\ge C_1
        \int_{\Sigma_\lambda^-}\varphi(x)^2
        \left(
        \int_{\Sigma_\lambda\setminus\Sigma_\lambda^-}
        \frac{dy}{|x-y|^{n+sp}}
        \right)dx\\
         &\ge C_1
        \int_{\Sigma_\lambda^-}\varphi(x)^2
        \left(
        \int_{\Sigma_\lambda\setminus D}
        \frac{dy}{|x-y|^{n+sp}}
        \right)dx\\
        &\ge C_1
        |D|^{-sp/n}
        \int_{\Sigma_\lambda^-}\varphi^2\,dx.
\end{aligned}
\end{equation}
On the other hand, testing the first inequality in \eqref{eq:lemma6-assumption} with \(\varphi\), and using \eqref{eq:lemma6-c-lower-bound}, gives
\begin{equation}\label{eq:small-region-energy-upper}
        \left\langle
        (-\Delta)^s_pu-(-\Delta)^s_pu_\lambda,
        \varphi
        \right\rangle
        \le M\int_{\Sigma_\lambda^-}\varphi^2\,dx.
\end{equation}
Choose \(\delta>0\) so small that \(C_1\delta^{-sp/n}>M\). If \(|D|\le\delta\), then \eqref{eq:reflected-energy-small-support} and \eqref{eq:small-region-energy-upper} imply \(|\Sigma_\lambda^-|=0\), which is equivalent to \(u_\lambda\ge u\) a.e. in \(\Sigma_\lambda\).
\end{proof}

\section{Proof of Theorem~\ref{Liouville thm}}\label{sec:bounded-domain-proof}
\begin{proof}[Proof of Theorem~\ref{Liouville thm}]
Set
\begin{equation}\label{eq:bounded-f-Lipschitz-constant}
 L_u:=\sup_{\substack{a,b\in[0,\|u\|_{L^\infty(B_1)}]\\ a\ne b}}
 \frac{|f(a)-f(b)|}{|a-b|}<\infty.
\end{equation}
Since \(f(u)\in L^\infty(B_1)\), \cite[Theorem~1.1]{MR4788673} shows that \(u\) has a continuous representative on \(\mathbb R^n\), after extending it by zero outside \(B_1\).
\cite[Theorem~2.6]{MR4604546} yields
\begin{equation}\label{eq:bounded-general-f-positivity}
        u>0\qquad\text{in }B_1.
\end{equation}

\noindent \textit{Step 1: For $\lambda>-1$ close enough to $-1$, $u_\lambda\ge u$ in $\Omega_\lambda$.}

Fix \(\lambda\in(-1,0)\), and let \(\Omega_\lambda^-\) be given by \eqref{def:Omega}. If \(\Omega_\lambda^-=\varnothing\), then $u_\lambda\ge u$ in $\Omega_\lambda$ holds for this \(\lambda\). Hence we may assume \(\Omega_\lambda^-\ne\varnothing\). Then
\begin{equation*}
\begin{cases}
 (-\Delta)^s_pu_\lambda-(-\Delta)^s_pu
        -\bigl(f(u_\lambda)-f(u)\bigr)=0
        & \text{weakly in }\Omega_\lambda^-,\\
u_\lambda-u\ge0
        & \text{a.e. in }\Sigma_\lambda\setminus \Omega_\lambda^-.
\end{cases}
\end{equation*}
For \(x\in\Omega_\lambda^-\), define
\[
c_\lambda(x):=-\frac{f(u_\lambda(x))-f(u(x))} {u_\lambda(x)-u(x)}.
\]
Then
\begin{equation}\label{eq:step1-p-less-2-srp-weak-equation}
        (-\Delta)^s_pu_\lambda-(-\Delta)^s_pu
        +c_\lambda(x)(u_\lambda-u)=0
        \qquad\text{weakly in }\Omega_\lambda^-,
\end{equation}
and the Lipschitz continuity of \(f\) on the range of \(u\) gives
\begin{equation}\label{eq:step1-p-less-2-srp-c-lower}
        c_\lambda(x)\ge-L_u
        \quad\text{for }x\in\Omega_\lambda^-.
\end{equation}
Then \eqref{eq:step1-p-less-2-srp-weak-equation} and \eqref{eq:step1-p-less-2-srp-c-lower} are exactly the hypotheses of the small region principle, Theorem~\ref{lem:small-region-p-less-2}, on the domain \(\Omega_\lambda^-\). Let $\delta$ be the constant given by Theorem~\ref{lem:small-region-p-less-2}. Choose $\lambda>-1$ so close to $-1$ that $ |\Omega_\lambda|\le\delta$. Applying Theorem~\ref{lem:small-region-p-less-2} gives $u_\lambda\ge u$ in $\Omega_\lambda$. This completes Step~1.

Define
\begin{equation*}
\lambda_0:=\sup\Bigl\{
\lambda\in(-1,0]:
 u_\mu\ge u \ \text{a.e. in }\Omega_\mu
 \text{ for every }\mu\in(-1,\lambda]
\Bigr\}.
\end{equation*}

\noindent\textit{Step 2: \(\lambda_0=0\).}

Suppose, by contradiction, that \(\lambda_0<0\). Then
\begin{equation}\label{eq:step2-no-smp-order-at-lambda0}
        u_{\lambda_0}\ge u
        \qquad\text{in }\Omega_{\lambda_0}.
\end{equation}
We claim that
\begin{equation}\label{eq:step2-no-smp-not-identical}
        u_{\lambda_0}-u\not\equiv0
        \qquad\text{in }\Omega_{\lambda_0}.
\end{equation}
Indeed, if \(u_{\lambda_0}-u\equiv0\) in \(\Omega_{\lambda_0}\), continuity up to the boundary gives
\[
u_{\lambda_0}=u=0 \qquad\text{on } \partial B_1\cap\{x_1<\lambda_0\}.
\]
For any \(z\in\partial B_1\cap\{x_1<\lambda_0\}\),
\[
0=u_{\lambda_0}(z)=u(z^{\lambda_0}).
\]
Notice that \(z^{\lambda_0}\in B_1\). This contradicts \eqref{eq:bounded-general-f-positivity}. Hence \eqref{eq:step2-no-smp-not-identical} holds.

It follows from \eqref{eq:step2-no-smp-order-at-lambda0}, \eqref{eq:step2-no-smp-not-identical}, and continuity that there exist \(x_*\in\Omega_{\lambda_0}\), \(r>0\), and \(\eta>0\) such that \(B_r(x_*)\Subset\Omega_{\lambda_0}\) and
\begin{equation*}
        u_{\lambda_0}-u\ge2\eta
        \qquad\text{in }\overline{B_r(x_*)}.
\end{equation*}
Extend the continuous representative of \(u\) by zero to all of \(\mathbb R^n\). For \(t\ge0\), define
\[
\begin{aligned}
 \omega_\infty(t)
 &:=\sup_{\substack{x,y\in\mathbb R^n,\,|x-y|\le t}}
       |u(x)-u(y)|.
\end{aligned}
\]
The zero extension is uniformly continuous on \(\mathbb R^n\), and hence \(\omega_\infty(t)\to0\) as \(t\downarrow0\).  For \(\varepsilon_0>0\) small enough, we may assume that $ 0<\varepsilon_0<-\lambda_0$ and $\omega_\infty(2\varepsilon_0)\le\eta$. Then, for every \(\lambda\in(\lambda_0,\lambda_0+\varepsilon_0)\),
\begin{equation}\label{eq:step2-no-smp-gap-near-lambda0}
        u_\lambda-u\ge\eta
        \qquad\text{in }\overline{B_r(x_*)}.
\end{equation}
Fix now \(\lambda\in(\lambda_0,\lambda_0+\varepsilon_0)\) and let \(\psi\) and \(\varphi\) be defined by \eqref{eq def of test func}. %
Then \eqref{eq:step2-no-smp-gap-near-lambda0} yields
\begin{equation}\label{eq:step2-no-smp-ball-signs}
        \psi\le-\eta,
        \qquad
        \varphi=0
        \qquad\text{in }B_r(x_*).
\end{equation}
If \(x_1<\lambda_0\), then \eqref{eq:step2-no-smp-order-at-lambda0} gives
\[
\begin{aligned}
\varphi(x)
&=\bigl(u(x)-u(x^\lambda)\bigr)^+ \le
\bigl(u(x^{\lambda_0})-u(x^\lambda)\bigr)^+\le
\left|u(x^{\lambda_0})-u(x^\lambda)\right| \le
\omega_\infty\bigl(2(\lambda-\lambda_0)\bigr).
\end{aligned}
\]
If \(\lambda_0\le x_1<\lambda\), then
\[
\begin{aligned}
\varphi(x)
&\le \left|u(x)-u(x^\lambda)\right|\le \omega_\infty\bigl(2(\lambda-x_1)\bigr)\le \omega_\infty\bigl(2(\lambda-\lambda_0)\bigr).
\end{aligned}
\]
Thus
\begin{equation}\label{eq:step2-no-smp-linfty-smallness}
        \|\varphi\|_{L^\infty(\Sigma_\lambda)}
        \le
        \omega_\infty\bigl(2(\lambda-\lambda_0)\bigr)\longrightarrow0
        \qquad\text{as }\lambda\downarrow\lambda_0.
\end{equation}
We record the estimate needed in the remainder of this step.
\begin{lemma}
\label{lem:reflected-energy-coercivity}
There exists \(c_0>0\), independent of \(\lambda\in(\lambda_0,\lambda_0+\varepsilon_0)\), such that
\begin{equation}\label{eq:step2-no-smp-energy-lower-final}
        \left\langle
        (-\Delta)^s_pu-(-\Delta)^s_pu_\lambda,
        \varphi
        \right\rangle
        \ge c_0\int_{\Omega_\lambda^-}\varphi^\theta\,dx,
        \qquad
        \theta:=\frac{sp+2}{sp+1}\in(1,2).
\end{equation}
\end{lemma}

\begin{proof}[Proof of Lemma~\ref{lem:reflected-energy-coercivity}]
Using the decomposition from the proof of Lemma~\ref{lem:reflection-inequality}, write
\[
\left\langle (-\Delta)^s_pu-(-\Delta)^s_pu_\lambda, \varphi \right\rangle =I_1+I'_2+I''_2.
\]
By \eqref{eq:reflected-energy-I2},
\begin{equation}\label{eq:reflected-energy-hardy-term}
\begin{aligned}
 I'_2+I''_2
 &\ge C_1\iint_{\Sigma_\lambda\times\Sigma_\lambda}
 \frac{\varphi(x)^2+\varphi(y)^2}{|x-y^\lambda|^{n+sp}}\,dx\,dy\ge C_1\iint_{\Sigma_\lambda\times\Sigma_\lambda}
 \frac{\varphi(x)^2}{|x-y^\lambda|^{n+sp}}\,dx\,dy\\
 &=C_1\int_{\Omega_\lambda^-}\varphi(x)^2
 \left(\int_{\Sigma_\lambda}
 \frac{dy}{|x-y^\lambda|^{n+sp}}\right)dx\\
 &\ge a_0\int_{\Omega_\lambda^-}
 (\lambda-x_1)^{-sp}\varphi(x)^2\,dx,
\end{aligned}
\end{equation}
where \(a_0>0\) is independent of \(\lambda\).

We next estimate \(I_1\). For \(x\in\Omega_\lambda^-\) and a.e. \(y\in B_r(x_*)\), \eqref{eq:step2-no-smp-ball-signs} and \eqref{eq:23} give
\begin{equation}\label{eq:32}
\begin{aligned}
 &\bigl[G(u(x)-u(y))-G(u_\lambda(x)-u_\lambda(y))\bigr]
 \bigl(\varphi(x)-\varphi(y)\bigr)\\
 &\quad\ge C_p\|u\|_{L^\infty(B_1)}^{p-2}
 \bigl(\psi(x)-\psi(y)\bigr)
 \bigl(\varphi(x)-\varphi(y)\bigr)\\
 &\quad=C\bigl(\psi(x)-\psi(y)\bigr)\varphi(x)
 \ge C\eta\varphi(x),
\end{aligned}
\end{equation}
where \(C>0\) is independent of \(\lambda\). By the mean value theorem, for some \(\xi\in\bigl(|x-y|^2,|x-y^\lambda|^2\bigr)\),
\[
\begin{aligned}
 \frac1{|x-y|^{n+sp}}-\frac1{|x-y^\lambda|^{n+sp}}
 =C(n,s,p)\xi^{-(n+sp)/2-1}
 (\lambda-x_1)(\lambda-y_1).
\end{aligned}
\]
Since \(\xi\le |x-y^\lambda|^2\le16\),
\begin{equation}\label{eq:reflected-energy-kernel-linear}
 \frac1{|x-y|^{n+sp}}-\frac1{|x-y^\lambda|^{n+sp}}
 \ge C(n,s,p)
 \left(\lambda_0-\sup_{z\in B_r(x_*)}z_1\right)
 (\lambda-x_1).
\end{equation}
The integrand in the definition \eqref{eq:I1} of \(I_1\) is nonnegative. Therefore, by \eqref{eq:32} and \eqref{eq:reflected-energy-kernel-linear},
\begin{equation}\label{eq:reflected-energy-anchor-term}
\begin{aligned}
 I_1
 &\ge C\eta
 \int_{\Omega_\lambda^-}\int_{B_r(x_*)}
 \left(
 \frac1{|x-y|^{n+sp}}-\frac1{|x-y^\lambda|^{n+sp}}
 \right)\varphi(x)\,dy\,dx\\
 &\ge b_0\int_{\Omega_\lambda^-}
 (\lambda-x_1)\varphi(x)\,dx,
\end{aligned}
\end{equation}
where \(b_0>0\) is independent of \(\lambda\).

Combining Lemma~\ref{lem:two-scale-algebraic-inequality}, \eqref{eq:reflected-energy-hardy-term}, and \eqref{eq:reflected-energy-anchor-term}, we obtain
\begin{equation*}
\begin{aligned}
 \left\langle
 (-\Delta)^s_pu-(-\Delta)^s_pu_\lambda,\varphi
 \right\rangle
 &\ge\int_{\Omega_\lambda^-}
 \left[a_0(\lambda-x_1)^{-sp}\varphi^2
 +b_0(\lambda-x_1)\varphi\right]dx\\
 &\ge c_0\int_{\Omega_\lambda^-}\varphi^\theta\,dx.
\end{aligned}
\end{equation*}
\end{proof}

Using the weak equations for \(u\) and \(u_\lambda\), together with \eqref{eq:bounded-f-Lipschitz-constant}, we obtain
\begin{equation}\label{eq:step2-no-smp-energy-upper}
\begin{aligned}
        \left\langle
        (-\Delta)^s_pu-(-\Delta)^s_pu_\lambda,
        \varphi
        \right\rangle
        &=\int_{\Omega_\lambda^-}
        \bigl(f(u)-f(u_\lambda)\bigr)\varphi\,dx\le L_u\int_{\Omega_\lambda^-}\varphi^2\,dx.
\end{aligned}
\end{equation}
Combining \eqref{eq:step2-no-smp-energy-upper} and \eqref{eq:step2-no-smp-energy-lower-final}, we obtain
\begin{equation}\label{eq:step2-no-smp-absorption}
\begin{aligned}
        c_0\int_{\Omega_\lambda^-}\varphi^\theta\,dx
        &\le L_u\int_{\Omega_\lambda^-}\varphi^2\,dx\le
        L_u\|\varphi\|_{L^\infty(\Sigma_\lambda)}^{2-\theta}
        \int_{\Omega_\lambda^-}\varphi^\theta\,dx.
\end{aligned}
\end{equation}
By \eqref{eq:step2-no-smp-linfty-smallness}, after decreasing \(\varepsilon_0\) once more, we have
\[
L_u\|\varphi\|_{L^\infty(\Sigma_\lambda)}^{2-\theta} <c_0 \qquad \text{for every } \lambda\in(\lambda_0,\lambda_0+\varepsilon_0).
\]
It follows from \eqref{eq:step2-no-smp-absorption} that $\varphi=0$ a.e. in $\Sigma_\lambda$, and therefore $u_\lambda\ge u$ a.e. in $\Omega_\lambda$ for every \(\lambda\in(\lambda_0,\lambda_0+\varepsilon_0)\). This contradicts the supremum definition of \(\lambda_0\). Consequently, $\lambda_0=0$. Repeating the same argument in the opposite direction, and then in every direction after a rotation, gives symmetry with respect to the origin. Since the direction is arbitrary, \(u\) is radially symmetric and radially nonincreasing in \(B_1\).

\end{proof}

\section{Proof of Theorem~\ref{thm:general}}\label{sec:whole-space-RD-proof}
\label{sec:whole-space-reflection-defect}

In this section we consider nonnegative weak solutions $u$ of \eqref{eq:whole-space-lane-emden}. The following lemma provides the lower bounds for the reflected term that are needed in the moving plane method.

\begin{lemma}
\label{lem:whole-space-K-lower}
There exists \(c_H=c_H(n,s,p)>0\) such that, for every \(\lambda\in\mathbb R\) and \(x\in\Sigma_\lambda^-\),
\begin{equation}\label{eq:whole-space-K-strip}
 \int_{\Sigma_\lambda}
 \frac{
 \bigl(
  |u(x)-u(y)|+|u_\lambda(x)-u(y)|
 \bigr)^{p-2}}
 {|x-y^\lambda|^{n+sp}}\,dy
 \ge c_H(\lambda-x_1)^{-sp}u(x)^{p-2}.
\end{equation}
If \(\lambda\le0\), then
\begin{equation}\label{eq:whole-space-K-far-negative-plane}
 \int_{\Sigma_\lambda}
 \frac{
 \bigl(
  |u(x)-u(y)|+|u_\lambda(x)-u(y)|
 \bigr)^{p-2}}
 {|x-y^\lambda|^{n+sp}}\,dy
 \ge c_H|x|^{-sp}u(x)^{p-2},
 \qquad\text{for }x\in\Sigma_\lambda^-.
\end{equation}
Let \(I\subset\mathbb R\) be a bounded interval and set $L_I:=\max\left\{1,\sup_{\mu\in I}|\mu|\right\}$. Then, for every \(\lambda\in I\) and every \(x\in\Sigma_\lambda^-\) with \(|x|\ge L_I\), there exists \(c_H'=c_H'(n,s,p)>0\) such that
\begin{equation}\label{41}
   \int_{\Sigma_\lambda}
 \frac{
 \bigl(
 |u(x)-u(y)|+|u_\lambda(x)-u(y)|
 \bigr)^{p-2}}
 {|x-y^\lambda|^{n+sp}}\,dy
 \ge c_H'|x|^{-sp}u(x)^{p-2}.
\end{equation}
\end{lemma}

\begin{proof}
For every \(\lambda\in\mathbb R\) and \(x\in\Sigma_\lambda^-\), take \(a=\lambda-x_1\), and set \(z:=x-\frac32ae_1\). Then \(B_a(z)\subset\Sigma_\lambda\) and \(x\in B_{2a}(z)\). Applying \cite[Lemma~3.3]{liu2025nontrivial} gives
\begin{equation}\label{eq:whole-space-average-u}
        \frac1{|B_a(z)|}\int_{B_a(z)}u(y)^{p-1}\,dy \le C_H(\operatorname*{inf}_{B_{2a}(z)}u)^{p-1}
        \le C u(x)^{p-1}.
\end{equation}
Using \eqref{eq:whole-space-average-u}, we obtain
\begin{equation}\label{eq:42}
\begin{aligned}
 & \frac1{|B_a(z)|}\int_{B_a(z)}
 \bigl(
 |u(x)-u(y)|+|u_\lambda(x)-u(y)|
 \bigr)^{p-1}\,dy   \\
   \le & \frac1{|B_a(z)|}\int_{B_a(z)}
 \bigl(
 |2u(x)+2u(y)|
 \bigr)^{p-1}\,dy  \le C u(x)^{p-1}.
 \end{aligned}
\end{equation}
Jensen's inequality gives
\[
\begin{aligned}
 &\frac1{|B_a(z)|}\int_{B_a(z)}
 \bigl(
 |u(x)-u(y)|+|u_\lambda(x)-u(y)|
 \bigr)^{p-2}\,dy\\
 &\quad\ge
 \left(
 \frac1{|B_a(z)|}\int_{B_a(z)}
 \bigl(
 |u(x)-u(y)|+|u_\lambda(x)-u(y)|
 \bigr)^{p-1}\,dy
 \right)^{\frac{p-2}{p-1}}
 \ge c u(x)^{p-2},
\end{aligned}
\]
where the last inequality follows from \eqref{eq:42}. For \(y\in B_a(z)\),
\[
\begin{aligned}
 |x-y^\lambda|
 &\le |x-x^\lambda|+|x^\lambda-y^\lambda|=2a+|x-y|\le2a+|x-z|+|z-y|\le\frac92a.
\end{aligned}
\]
Consequently,
\[
\begin{aligned}
 &\int_{\Sigma_\lambda}
 \frac{
 \bigl(
 |u(x)-u(y)|+|u_\lambda(x)-u(y)|
 \bigr)^{p-2}}
 {|x-y^\lambda|^{n+sp}}\,dy\\
 &\quad\ge
 \left(\frac92a\right)^{-n-sp}
 \int_{B_a(z)}
 \bigl(
 |u(x)-u(y)|+|u_\lambda(x)-u(y)|
 \bigr)^{p-2}\,dy\\
 &\quad\ge
 c_Ha^{-sp}u(x)^{p-2}
 =c_H(\lambda-x_1)^{-sp}u(x)^{p-2},
\end{aligned}
\]
which proves \eqref{eq:whole-space-K-strip}. If \(\lambda\le0\) and \(x\in\Sigma_\lambda^-\), then \(\lambda-x_1\le |x|\). Thus \eqref{eq:whole-space-K-far-negative-plane} follows directly from \eqref{eq:whole-space-K-strip}. If \(|x|\ge L_I\), then $\lambda-x_1\le|\lambda|+|x_1|\le2|x|$. Then \eqref{41} again follows from \eqref{eq:whole-space-K-strip}.
\end{proof}

\begin{proof}[Proof of Theorem~\ref{thm:general}]
We first move the plane in the fixed direction \(e_1\).

\smallskip
\noindent\emph{Step 1: $u_\lambda\ge u$ in $\Sigma_\lambda$ for $\lambda$ sufficiently negative.}

Suppose otherwise that $\Sigma_\lambda^-\neq \varnothing$. Then we can choose \(\varepsilon>0\) such that \(A_{\lambda,\varepsilon}\ne\varnothing\). Let $c_q=\max\{1,q\}$. Then
\begin{equation}\label{eq:whole-space-reaction-upper}
\begin{aligned}
 &\left\langle
 (-\Delta)^s_pu-(-\Delta)^s_pu_\lambda,
 \varphi_\varepsilon
 \right\rangle
 \le c_q\int_{A_{\lambda,\varepsilon}}
 u^{q-1}\psi
 \varphi_\varepsilon\,dx.
\end{aligned}
\end{equation}
By the decay condition \hyperref[eq:RD-prime]{\textup{(D)}}, there exists \(R_0>1\) such that
\begin{equation}\label{eq:whole-space-start-smallness}
        |x|^{sp}u(x)^{q-p+1}\le c_*
\end{equation}
for every \(\lambda\in\mathbb R\) and every \(x\in\Sigma_\lambda^-\) with \(|x|\ge R_0\). For \(\lambda\le-R_0\), \eqref{c}, \eqref{eq:whole-space-reflected-energy}, \eqref{eq:whole-space-K-far-negative-plane} and \eqref{eq:whole-space-start-smallness} imply
\[
\begin{aligned}
 &\left\langle
 (-\Delta)^s_pu-(-\Delta)^s_pu_\lambda,
 \varphi_\varepsilon
 \right\rangle \ge c_pc_H \int_{A_{\lambda,\varepsilon}}
 |x|^{-sp}u^{p-2}\psi
 \varphi_\varepsilon\,dx \\
 &\quad
\ge 2c_qc_*\int_{A_{\lambda,\varepsilon}}
 |x|^{-sp}u^{p-2}\psi\varphi_\varepsilon\,dx
 \\
 &\quad\ge 2c_q
\int_{A_{\lambda,\varepsilon}}
 u^{q-1}\psi
\varphi_\varepsilon\,dx,
\end{aligned}
\]
contradicting \eqref{eq:whole-space-reaction-upper}. Thus $\Sigma_\lambda^-=\varnothing$.

Define
\begin{equation*}
\begin{aligned}
 \lambda_0:=\sup\bigl\{\lambda\in\mathbb R:
&u_\mu\ge u\text{ in }\Sigma_\mu
 \text{ for every }\mu\le\lambda\bigr\}.
\end{aligned}
\end{equation*}

\smallskip
\noindent\emph{Step 2: If \(\lambda_0<+\infty\), then $u_{\lambda_0}\equiv u$.}

Suppose not. Then there exist \(B_r(x_*)\Subset\Sigma_{\lambda_0}\) and \(\eta>0\) such that $u_{\lambda_0}-u\ge2\eta$ in $B_r(x_*)$. By local uniform continuity, for some \(\delta_0>0\),
\begin{equation*}
        \psi\le-\eta
        \qquad\text{in }B_r(x_*)
\end{equation*}
whenever \(\lambda\in[\lambda_0,\lambda_0+\delta_0]\). Choose \(R\) so large that
\[
R\ge \max \{2(|\lambda_0|+\delta_0+1),R_0,|x_*|+r+1\},
 \]
then \eqref{41} holds whenever \(\lambda\in[\lambda_0,\lambda_0+\delta_0]\) and \(|x|\ge R\), and \eqref{eq:whole-space-start-smallness} holds
for every \(\lambda\in[\lambda_0,\lambda_0+\delta_0]\) and every \(x\in\Sigma_\lambda^-\) with \(|x|\ge R\). Choose \(R_1>R\) so large that $x^\lambda\in B_{R_1}$ for every $x\in B_R\cup B_r(x_*)$ and $\lambda\in[\lambda_0,\lambda_0+\delta_0]$. Choose \(\rho\in(0,1)\) so small that
\begin{equation}\label{eq:whole-space-strip-smallness}
        \rho^{sp}\|u\|_{L^\infty(B_{R_1})}^{q-p+1}
        \le c_*.
\end{equation}
Fix \(\lambda\in(\lambda_0,\lambda_0+\delta_0]\) and \(\varepsilon>0\). We record the estimate needed in the remainder of this step.
\begin{lemma}\label{whole-lem}
There exists \(b_*>0\), independent of \(\lambda\) and \(\varepsilon\), such that
\begin{equation}\label{eq:whole-space-far-strip-coefficient}
 \begin{aligned}
 &\left\langle
 (-\Delta)^s_pu-(-\Delta)^s_pu_\lambda,
 \varphi_\varepsilon
 \right\rangle\\
 &\quad\ge
 b_*\int_{A_{\lambda,\varepsilon}\cap B_R
 \cap\{\lambda-x_1\ge\rho\}}
 \varphi_\varepsilon\,dx\\
 &\qquad+
 2c_q\int_{A_{\lambda,\varepsilon}\cap\{|x|\ge R\}}
 u^{q-1}\psi
 \varphi_\varepsilon\,dx+2c_q\int_{A_{\lambda,\varepsilon}\cap B_R
       \cap\{\lambda-x_1<\rho\}}
 u^{q-1}\psi
 \varphi_\varepsilon\,dx.
\end{aligned}
\end{equation}
\end{lemma}
\begin{proof}[Proof of Lemma~\ref{whole-lem}]
Recall the decomposition $\left\langle (-\Delta)^s_pu-(-\Delta)^s_pu_\lambda, \varphi_\varepsilon \right\rangle=J_1+J'_2+J''_2$. We first estimate $J_1$. Restrict the integrand to \(x\in A_{\lambda,\varepsilon}\cap B_R \cap\{\lambda-x_1\ge\rho\}\) and \(y\in B_r(x_*)\).
As in \eqref{eq:32}, we have
\[
\bigl[ G(u(x)-u(y))-G(u_\lambda(x)-u_\lambda(y)) \bigr](\varphi_\varepsilon(x)-\varphi_\varepsilon(y)) \ge c_1\eta\varphi_\varepsilon(x),
\]
where $c_1>0$ is independent of $\lambda$ and $\varepsilon$. By the mean value theorem, for some \(\xi\in\bigl(|x-y|^2,|x-y^\lambda|^2\bigr)\),
\[
\begin{aligned}
 \frac1{|x-y|^{n+sp}}-\frac1{|x-y^\lambda|^{n+sp}}
 =C(n,s,p)\xi^{-(n+sp)/2-1}
 (\lambda-x_1)(\lambda-y_1).
\end{aligned}
\]
Since \(\xi\le|x-y^\lambda|^2\le(R+R_1)^2\),
\[
\frac1{|x-y|^{n+sp}}-\frac1{|x-y^\lambda|^{n+sp}} \ge C \rho\, \dist\bigl(B_r(x_*),T_{\lambda_0}\bigr)=:c_2>0,
\]
where \(c_2\) is independent of $\lambda$ and $\varepsilon$. Consequently,
\[
\begin{aligned}
 J_1
 &\ge b_*
 \int_{A_{\lambda,\varepsilon}\cap B_R
       \cap\{\lambda-x_1\ge\rho\}}
 \varphi_\varepsilon(x)\,dx,
\end{aligned}
\]
where $b_*:=c_1 c_2 \eta|B_r(x_*)|>0$. This constant is independent of \(\lambda\) and \(\varepsilon\). Combining the estimate of \(J'_2\) in \eqref{eq:whole-space-reflected-energy} and \(J''_2\ge0\), we obtain
\begin{equation}\label{eq:whole-space-reflected-energy-anchor}
\begin{aligned}
 &\left\langle
 (-\Delta)^s_pu-(-\Delta)^s_pu_\lambda,
 \varphi_\varepsilon
 \right\rangle\\
 &\quad\ge
 b_*\int_{A_{\lambda,\varepsilon}\cap B_R
 \cap\{\lambda-x_1\ge\rho\}}
 \varphi_\varepsilon\,dx\\
 &\qquad+
 c_p\int_{A_{\lambda,\varepsilon}}
 \psi(x)\varphi_\varepsilon(x)
 \int_{\Sigma_\lambda}
 \frac{
 \bigl(
 |u(x)-u(y)|+|u_\lambda(x)-u(y)|
 \bigr)^{p-2}}
 {|x-y^\lambda|^{n+sp}}\,dy\,dx.
\end{aligned}
\end{equation}
For \(x\in A_{\lambda,\varepsilon}\cap B_R \cap\{\lambda-x_1<\rho\}\), \eqref{eq:whole-space-K-strip} and \eqref{eq:whole-space-strip-smallness} give
\begin{equation}\label{50}
  \begin{aligned}
&c_p\int_{\Sigma_\lambda}
 \frac{
 \bigl(
  |u(x)-u(y)|+|u_\lambda(x)-u(y)|
 \bigr)^{p-2}}
 {|x-y^\lambda|^{n+sp}}\,dy \ge c_pc_H(\lambda-x_1)^{-sp}u(x)^{p-2}\\
& \qquad \ge c_pc_H\rho^{-sp}u(x)^{p-2}
 \ge2c_q c_*\rho^{-sp}u(x)^{p-2}\\
& \qquad \ge2c_q\|u\|_{L^\infty(B_{R_1})}^{q-p+1}
 u(x)^{p-2} \ge 2 c_q u(x)^{q-1}.
\end{aligned}
\end{equation}
For $x\in A_{\lambda,\varepsilon}\cap\{|x|\ge R\}$, \eqref{41} and \eqref{eq:whole-space-start-smallness} yield
\begin{equation}\label{51}
\begin{aligned}
&c_p\int_{\Sigma_\lambda}
 \frac{
 \bigl(
  |u(x)-u(y)|+|u_\lambda(x)-u(y)|
 \bigr)^{p-2}}
 {|x-y^\lambda|^{n+sp}}\,dy \ge c_pc_H'|x|^{-sp}u(x)^{p-2}\\
&\quad\ge2 c_q c_*|x|^{-sp}u(x)^{p-2}
 \ge2c_q u(x)^{q-1}.
\end{aligned}
\end{equation}
Combining \eqref{eq:whole-space-reflected-energy-anchor}, \eqref{50} and \eqref{51}, we obtain \eqref{eq:whole-space-far-strip-coefficient}.
\end{proof}
Define
\[
\omega_R(t):=\sup_{\substack{x,y\in\overline{B_R}\\ |x-y|\le t}} |u(x)-u(y)|, \qquad R>0.
\]
As in the proof of \eqref{eq:step2-no-smp-linfty-smallness}, we have
\begin{equation}\label{eq:whole-space-core-smallness}
        \psi(x)^+
        \le\omega_{R_1}
        \bigl(2(\lambda-\lambda_0)\bigr)
        \qquad\text{for }x\in\Sigma_\lambda\cap B_R.
\end{equation}
Lemma~\ref{lem:power-difference-all-q} and \eqref{eq:whole-space-core-smallness} yield
\begin{equation}\label{eq:whole-space-core-reaction-holder}
\begin{aligned}
 &\int_{A_{\lambda,\varepsilon}\cap B_R \cap\{\lambda-x_1\ge\rho\}}
 \bigl(u^q-u_\lambda^q\bigr)\varphi_\varepsilon\,dx\le
 C\,
 \omega_{R_1}\bigl(2(\lambda-\lambda_0)\bigr)^{\min\{1,q\}}
 \int_{A_{\lambda,\varepsilon}\cap B_R \cap\{\lambda-x_1\ge\rho\}}
 \varphi_\varepsilon\,dx,
\end{aligned}
\end{equation}
where $C=C(q,\|u\|_{L^\infty(B_{R_1})})>0$. Comparing \eqref{eq:whole-space-reaction-upper} with \eqref{eq:whole-space-far-strip-coefficient}, and subtracting from both sides the integral over $A_{\lambda,\varepsilon}\cap\{|x|\ge R\}$ and $A_{\lambda,\varepsilon}\cap B_R\cap\{\lambda-x_1<\rho\}$, gives
\[
\begin{aligned}
 b_*\int_{A_{\lambda,\varepsilon}\cap B_R
       \cap\{\lambda-x_1\ge\rho\}}
 \varphi_\varepsilon\,dx
 &+c_q \left(\int_{A_{\lambda,\varepsilon}\cap\{|x|\ge R\}}+ \int_{A_{\lambda,\varepsilon}\cap B_R
       \cap\{\lambda-x_1<\rho\}}\right)
\bigg( u^{q-1}\psi\varphi_\varepsilon\,dx\bigg)\\
 &\le  C\,
 \omega_{R_1}\bigl(2(\lambda-\lambda_0)\bigr)^{\min\{1,q\}}
 \int_{A_{\lambda,\varepsilon}\cap B_R
       \cap\{\lambda-x_1\ge\rho\}}
 \varphi_\varepsilon\,dx,
\end{aligned}
\]
for every \(\lambda\in[\lambda_0,\lambda_0+\delta_0]\).
Thus
\[
\begin{aligned}
 &\Bigl[
 b_*-C
 \omega_{R_1}\bigl(2(\lambda-\lambda_0)\bigr)^{\min\{1,q\}}
 \Bigr]
 \int_{A_{\lambda,\varepsilon}\cap B_R \cap\{\lambda-x_1\ge\rho\}}
 \varphi_\varepsilon\,dx\\
 &\quad+c_q\int_{A_{\lambda,\varepsilon}\cap\{|x|\ge R\}}
 u^{q-1}\psi\varphi_\varepsilon\,dx\\
 &\quad+c_q\int_{A_{\lambda,\varepsilon}\cap B_R
       \cap\{\lambda-x_1<\rho\}}
 u^{q-1}\psi\varphi_\varepsilon\,dx
 \le0.
\end{aligned}
\]
By choosing \(\delta_0\) small enough, $C \omega_{R_1}\bigl(2(\lambda-\lambda_0)\bigr)^{\min\{1,q\}}$ is smaller than \(b_*\). Therefore each of the three integrals is zero, which implies \(|A_{\lambda,\varepsilon}|=0\). Hence \(A_{\lambda,\varepsilon}=\varnothing\). Since \(\varepsilon>0\) is arbitrary, $\Sigma_\lambda^-=\varnothing$ for every \(\lambda>\lambda_0\) sufficiently close to \(\lambda_0\), contradicting the definition of \(\lambda_0\).

\smallskip
\noindent\emph{Step 3: $\lambda_0<+\infty$.}

We claim that \(\lambda_0<+\infty\). Suppose, by contradiction, that \(\lambda_0=+\infty\). Then \(u_\lambda\ge u\) in \(\Sigma_\lambda\) for every \(\lambda\in\mathbb R\). Consequently, for every \(z\in e_1^\perp\), the function \(h_z(t):=u(z+te_1)\) is nondecreasing on \(\mathbb R\). We first show that \(h_z\) must be constant. Suppose otherwise. Then, for some \(z\in e_1^\perp\), there exist \(t_1<t_2\) such that \(h_z(t_1)<h_z(t_2)\). Since \(u\ge0\), this also implies \(h_z(t_2)>0\). For \(T>t_2\), let $x_T:=z+Te_1$, $\mu_T:=-\frac{T+t_1}{2}$. Here \(\mu_T\) is the parameter of a moving plane in the direction \(-e_1\). A direct calculation gives
\[
x_T\in\Sigma_{\mu_T,-e_1}, \qquad x_T^{\mu_T,-e_1}=z+t_1e_1,
\]
where $x^{\lambda,e}$ and $\Sigma_{\lambda,e}$ are defined in Section~\ref{sec:introduction}. Since \(h_z\) is nondecreasing,
\[
u(x_T)=h_z(T)\ge h_z(t_2)>h_z(t_1) =u(x_T^{\mu_T,-e_1}).
\]
Hence $x_T\in \Sigma^-_{\mu_T,-e_1}$.
Since \(|x_T|\to\infty\), it follows that
\[
|x_T|^{sp}u(x_T)^{q-p+1}  \ge  |x_T|^{sp} {h_z(t_2)}^{q-p+1} \longrightarrow+\infty,
\]
which contradicts the decay condition \hyperref[eq:RD-prime]{\textup{(D)}}, applied in the direction \(-e_1\). Therefore \(h_z\) is constant for every \(z\in e_1^\perp\), and hence
\begin{equation}\label{eq:whole-space-constant-on-e-lines}
        u(x+te_1)=u(x)
        \qquad
        \text{for every }x\in\mathbb R^n,\ t\in\mathbb R.
\end{equation}
There exist \(a,b\in\mathbb R^n\) such that $u(a)>u(b)$. Set $\tilde{e}:=\frac{b-a}{|b-a|}$. For \(T>0\), define \(a_T:=a+Te_1\), \(b_T:=b+Te_1\), and \(\nu_T:=(a_T\cdot\tilde{e}+b_T\cdot\tilde{e})/2\). Then
\[
a_T\in\Sigma_{\nu_T,\tilde{e}}, \qquad a_T^{\nu_T,\tilde{e}}=b_T.
\]
By \eqref{eq:whole-space-constant-on-e-lines},
\[
u(a_T)=u(a)>u(b)=u(b_T),
\]
and therefore \(a_T\in \Sigma^-_{\nu_T,\tilde{e}}\). Moreover, \(|a_T|\longrightarrow\infty\) and \(u(a_T)=u(a)>0\). As above, continuity shows that the strict inequality persists on a set of positive measure near \(a_T\). Hence
\[
\sup_{\nu\in\mathbb R}
 \operatorname*{sup}_{\substack{x\in\Sigma_{\nu,\tilde{e}}^-\\
|x|\ge |a_T|-1}} |x|^{sp}u(x)^{q-p+1} \longrightarrow+\infty,
\]
contradicting the decay condition \hyperref[eq:RD-prime]{\textup{(D)}}, applied in the direction \(\tilde{e}\). Thus \(\lambda_0<+\infty\), and Step~2 gives \(u_{\lambda_0}\equiv u\). The same proof applies, after a rotation, for each direction \(e\in \mathbb S^{n-1}\). By a geometric argument, one can verify that \(u\) is radially symmetric and radially nonincreasing about the point \((\lambda_0(e_1),\lambda_0(e_2),\cdots,\lambda_0(e_n))\).
\end{proof}

\section{Proof of Theorem~\ref{thm:critical}}\label{sec:critical-energy-proof}

We establish the decay condition \hyperref[eq:RD-prime]{\textup{(D)}} required
for applying the moving plane method in the critical case. In fact, we
obtain a stronger decay result.

\begin{lemma}
\label{lem:critical-finite-energy-decay}
Let \(n\ge2\), \(0<s<1\) and \(1<p\le2\). Let $0\le u\in \mathcal D^{s,p}(\mathbb R^n)\cap C(\mathbb R^n)$ be a weak solution of \eqref{eq:whole-space-critical-equation}. Then
\begin{equation}\label{eq:critical-scale-decay}
        \lim_{R\to\infty}
        \sup_{|x|\ge R}|x|^{\beta_c}u(x)=0,\qquad \text{where } \beta_c:=\frac{n-sp}{p}.
\end{equation}
\end{lemma}

\begin{proof}
Suppose that \eqref{eq:critical-scale-decay} fails. Then there exist \(\varepsilon>0\) and a sequence \(x_k\in\mathbb R^n\), with \(|x_k|\to\infty\), such that
\begin{equation}\label{eq:critical-decay-contradiction-sequence}
        |x_k|^{\beta_c}u(x_k)\ge\varepsilon.
\end{equation}
Set \(r_k:=|x_k|/4\) and \(m(x):=u(x)^{1/\beta_c}\). On \(\overline{B_{r_k}(x_k)}\), consider the continuous function \(x\mapsto m(x)\dist\bigl(x,\partial B_{r_k}(x_k)\bigr)\). It attains a positive maximum at some point \(y_k\in B_{r_k}(x_k)\). By maximality and \eqref{eq:critical-decay-contradiction-sequence},
\begin{equation}\label{eq:critical-point-selection-product}
  \begin{aligned}
 &m(y_k)\dist\bigl(y_k,\partial B_{r_k}(x_k)\bigr)\ge m(x_k)r_k
 =\frac14\bigl(|x_k|^{\beta_c}u(x_k)\bigr)^{1/\beta_c}
 \ge\frac{\varepsilon^{1/\beta_c}}4=:4\kappa.
\end{aligned}
\end{equation}
If \(|x-y_k|\le\kappa/m(y_k)\), then \eqref{eq:critical-point-selection-product} gives \(|x-y_k|\le \dist(y_k,\partial B_{r_k}(x_k))/4\). Consequently,
\[
\dist(x,\partial B_{r_k}(x_k))\ge \frac34\dist(y_k,\partial B_{r_k}(x_k)),
\]
and the maximality of \(y_k\) implies
\begin{equation}\label{eq:critical-point-selection-local-bound}
        m(x)\le\frac43m(y_k)<2m(y_k).
\end{equation}
Set $\rho_k:=u(y_k)^{-1/\beta_c}=m(y_k)^{-1}$, and define the rescaling
\[
v_k(z):=\frac{u(y_k+\rho_k z)}{u(y_k)}.
\]
By \eqref{eq:critical-point-selection-local-bound},
\begin{equation}\label{eq:critical-rescaled-local-bound}
        v_k(0)=1,
        \qquad
        0\le v_k\le2^{\beta_c}
        \quad\text{in }B_\kappa.
\end{equation}
The rescaled functions satisfy
\[
(-\Delta)^s_pv_k=v_k^{p_s^*-1} \qquad\text{weakly in }\mathbb R^n,
\]
and
\begin{equation}\label{eq:critical-rescaled-LP-invariance}
        \|v_k\|_{L^{p_s^*}(\mathbb R^n)}
        =\|u\|_{L^{p_s^*}(\mathbb R^n)}.
\end{equation}
For a measurable function \(v\), define
\[
\operatorname{Tail}(v;0,R) :=\left( R^{sp}\int_{\mathbb R^n\setminus B_R} \frac{|v(z)|^{p-1}}{|z|^{n+sp}}\,dz \right)^{1/(p-1)}.
\]
H\"older's inequality, with exponents \(p_s^*/(p-1)\) and \(p_s^*/(p_s^*-p+1)\), together with \eqref{eq:critical-rescaled-LP-invariance}, yields
\begin{equation}\label{eq:critical-rescaled-tail-bound}
\begin{aligned}
 &\int_{\mathbb R^n\setminus B_{\kappa/2}}
 \frac{v_k(z)^{p-1}}{|z|^{n+sp}}\,dz\le
 \|v_k\|_{L^{p_s^*}(\mathbb R^n)}^{p-1}
 \left(
 \int_{\mathbb R^n\setminus B_{\kappa/2}}
 |z|^{-\frac{(n+sp)p_s^*}{p_s^*-p+1}}\,dz
 \right)^{\frac{p_s^*-p+1}{p_s^*}}
 \le C,
\end{aligned}
\end{equation}
where \(C\) is independent of \(k\). Thus \(\operatorname{Tail}(v_k;0,\kappa/2)\) is uniformly bounded. In view of \eqref{eq:critical-rescaled-local-bound}--\eqref{eq:critical-rescaled-tail-bound}, the local H\"older estimate in \cite[Theorem~8.2]{Cozzi2017} gives constants \(\gamma\in(0,1)\) and \(C_0>0\), independent of \(k\), such that
\begin{equation}\label{eq:critical-rescaled-holder-bound}
        [v_k]_{C^\gamma(B_{\kappa/4})}\le C_0.
\end{equation}
Choose $0<\delta\le\frac\kappa4$ so small that \(C_0\delta^\gamma\le1/2\). Since \(v_k(0)=1\), \eqref{eq:critical-rescaled-holder-bound} implies $v_k\ge\frac12$ in $B_\delta$. Hence
\begin{equation}\label{eq:critical-rescaled-positive-mass}
        \int_{B_{\delta\rho_k}(y_k)}u^{p_s^*}\,dx=\int_{B_\delta}v_k^{p_s^*}\,dz
        \ge 2^{-p_s^*}|B_\delta|=:c_0>0.
\end{equation}
Observe that
\[
B_{\delta\rho_k}(y_k)\subset B_{r_k}(x_k)=B_{|x_k|/4}(x_k) \subset\mathbb R^n\setminus B_{3|x_k|/4}.
\]
Therefore, because \(u\in L^{p_s^*}(\mathbb R^n)\),
\[
\int_{B_{\delta\rho_k}(y_k)}u^{p_s^*}\,dx \le \int_{|x|\ge3|x_k|/4}u^{p_s^*}\,dx \longrightarrow0.
\]
This contradicts \eqref{eq:critical-rescaled-positive-mass}, and proves \eqref{eq:critical-scale-decay}.
\end{proof}

\begin{proof}[Proof of Theorem~\ref{thm:critical}]
For $q=p_s^*-1$, Lemma~\ref{lem:critical-finite-energy-decay} implies
\[
\begin{aligned}
        |x|^{sp}u(x)^{q-p+1}
        =\bigl(|x|^{\beta_c}u(x)\bigr)^{p_s^*-p}
        \longrightarrow0
        \qquad\text{as }|x|\to\infty.
\end{aligned}
\]
In particular, the decay condition \hyperref[eq:RD-prime]{\textup{(D)}} holds uniformly with respect to all directions and all positions of the moving plane. Since \(\mathcal D^{s,p}(\mathbb R^n)\subset W^{s,p}_{\mathrm{loc}}(\mathbb R^n)\cap L_{sp}^{p-1}(\mathbb R^n)\), all assumptions of Theorem~\ref{thm:general} are satisfied. The claimed radial symmetry and radial monotonicity therefore follow.

By Theorem~\ref{thm:general}, \(u\) is a radially symmetric and radially monotone solution of the critical Lane--Emden equation. Thus \eqref{eq:asy} follows from \cite[Remark~1.2]{BrascoMosconiSquassina2016}.
\end{proof}

\noindent{\bf Acknowledgements} Part of this work was completed while H. Yang was visiting The Chinese University of Hong Kong. He would like to thank Professor Juncheng Wei for his support and encouragement, as well as the Department of Mathematics and the Institute of Mathematical Sciences for their hospitality. The research of M. Xu is partially supported by the National Key R\&D Program of China 2025YFA1017600 and NSFC 12526202. The research of H. Yang is supported by NSFC 12671139 and 12301140.

\medskip

\noindent{\bf Data availability} No data were used in this study.

\medskip

\noindent{\bf Conflict of interest} There is no conflict of interest.

\providecommand{\bysame}{\leavevmode\hbox to3em{\hrulefill}\thinspace}
\providecommand{\MR}{\relax\ifhmode\unskip\space\fi MR }
\providecommand{\MRhref}[2]{%
  \href{http://www.ams.org/mathscinet-getitem?mr=#1}{#2}
}
\providecommand{\href}[2]{#2}

\bigskip

\noindent M. Xu

\noindent  School of Mathematical Sciences, Fudan University\\
Shanghai 200433, China \\[1mm] Email:  \textsf{meiqing\_xu@fudan.edu.cn}

\bigskip

\noindent   H. Yang

\noindent  School of Mathematical Sciences, Shanghai Jiao Tong University\\
Shanghai 200240, China \\[1mm] Email:  \textsf{hui-yang@sjtu.edu.cn}


\begin{thebibliography}{10}

\bibitem{BiswasRoySen2026}
A.~Biswas, S.~Roy, and A.~Sen, \emph{Strong comparison principle and symmetry
  results for the fractional {$p$}-{L}aplacian}, arXiv:2606.08559, 2026.

\bibitem{BiswasTopp2025}
A.~Biswas and E.~Topp, \emph{Lipschitz regularity of fractional
  {$p$}-{L}aplacian}, Ann. PDE \textbf{11} (2025), no.~2, Paper No. 27, 43 pp.

\bibitem{BogeleinEtAlCalcVar2025}
V.~B{\"o}gelein, F.~Duzaar, N.~Liao, G.~Molica Bisci, and R.~Servadei,
  \emph{Gradient regularity for {$(s,p)$}-harmonic functions}, Calc. Var.
  Partial Differential Equations \textbf{64} (2025), no.~8, Paper No. 253.

\bibitem{BogeleinEtAlJFA2025}
\bysame, \emph{Regularity for the fractional {$p$}-{L}aplace equation}, J.
  Funct. Anal. \textbf{289} (2025), no.~9, Paper No. 111078, 69 pp.

\bibitem{BG}
J.~P. Bouchaud and A.~Georges, \emph{Anomalous diffusion in disordered media:
  Statistical mechanisms, models and physical applications}, Phys. Rep.
  \textbf{195} (1990), 127--293.

\bibitem{BrascoLindgrenSchikorra2018}
L.~Brasco, E.~Lindgren, and A.~Schikorra, \emph{Higher {H}\"older regularity
  for the fractional {$p$}-{L}aplacian in the superquadratic case}, Adv. Math.
  \textbf{338} (2018), 782--846.

\bibitem{BrascoMosconiSquassina2016}
L.~Brasco, S.~Mosconi, and M.~Squassina, \emph{Optimal decay of extremals for
  the fractional {S}obolev inequality}, Calc. Var. Partial Differential
  Equations \textbf{55} (2016), no.~2, Paper No. 23, 32 pp.

\bibitem{CGS}
L.~Caffarelli, B.~Gidas, and J.~Spruck, \emph{Asymptotic symmetry and local
  behavior of semilinear elliptic equations with critical {Sobolev} growth},
  Comm. Pure Appl. Math. \textbf{42} (1989), no.~3, 271--297.

\bibitem{CaSi}
L.~Caffarelli and L.~Silvestre, \emph{An extension problem related to the
  fractional {L}aplacian}, Comm. Partial Differential Equations \textbf{32}
  (2007), 1245--1260.

\bibitem{CDQ}
D.~Cao, W.~Dai, and G.~Qin, \emph{Super poly-harmonic properties, {Liouville}
  theorems and classification of nonnegative solutions to equations involving
  higher-order fractional {Laplacians}}, Trans. Amer. Math. Soc. \textbf{374}
  (2021), no.~7, 4781--4813.

\bibitem{DiCastroKuusiPalatucci2014}
A.~Di Castro, T.~Kuusi, and G.~Palatucci, \emph{Nonlocal {H}arnack
  inequalities}, J. Funct. Anal. \textbf{267} (2014), no.~6, 1807--1836.
  \MR{3237774}

\bibitem{DiCastroKuusiPalatucci2016}
\bysame, \emph{Local behavior of fractional {$p$}-minimizers}, Ann. Inst. H.
  Poincar\'e C Anal. Non Lin\'eaire \textbf{33} (2016), no.~5, 1279--1299.
  \MR{3542614}

\bibitem{CL1}
W.~Chen and C.~Li, \emph{Classification of solutions of some nonlinear elliptic
  equations}, Duke Math. J. \textbf{63} (1991), no.~3, 615--622.

\bibitem{CL2}
\bysame, \emph{Maximum principles for the fractional {$p$}-{L}aplacian and
  symmetry of solutions}, Adv. Math. \textbf{335} (2018), 735--758.

\bibitem{CLL}
W.~Chen, C.~Li, and Y.~Li, \emph{A direct method of moving planes for the
  fractional {L}aplacian}, Adv. Math. \textbf{308} (2017), 404--437.

\bibitem{CLO3}
W.~Chen, C.~Li, and B.~Ou, \emph{Classification of solutions for an integral
  equation}, Comm. Pure Appl. Math. \textbf{59} (2006), no.~3, 330--343.

\bibitem{CLZ}
W.~Chen, Y.~Li, and R.~Zhang, \emph{A direct method of moving spheres on
  fractional order equations}, J. Funct. Anal. \textbf{272} (2017), no.~10,
  4131--4157.

\bibitem{CZ}
W.~Chen and J.~Zhu, \emph{Indefinite fractional elliptic problem and
  {L}iouville theorems}, J. Differential Equations \textbf{260} (2016),
  4758--4785.

\bibitem{Cozzi2017}
M.~Cozzi, \emph{Regularity results and {H}arnack inequalities for minimizers
  and solutions of nonlocal problems: A unified approach via fractional {D}e
  {G}iorgi classes}, J. Funct. Anal. \textbf{272} (2017), 4762--4837.

\bibitem{DLQSIAM}
W.~Dai, Z.~Liu, and G.~Qin, \emph{{Classification of nonnegative solutions to
  static Schr\"{o}dinger-Hartree-Maxwell type equations}}, SIAM J. Math. Anal.
  \textbf{53} (2021), no.~2, 1379--1410.

\bibitem{DaiLiuWang2022}
W.~Dai, Z.~Liu, and P.~Wang, \emph{Monotonicity and symmetry of positive
  solutions to fractional {$p$}-{L}aplacian equation}, Commun. Contemp. Math.
  \textbf{24} (2022), no.~6, Paper No. 2150005, 17 pp.

\bibitem{Damascelli1998}
L.~Damascelli, \emph{Comparison theorems for some quasilinear degenerate
  elliptic operators and applications to symmetry and monotonicity results},
  Ann. Inst. H. Poincar\'e C Anal. Non Lin\'eaire \textbf{15} (1998), no.~4,
  493--516.

\bibitem{DamascelliMerchanMontoroSciunzi2014}
L.~Damascelli, S.~Merch\'an, L.~Montoro, and B.~Sciunzi, \emph{Radial symmetry
  and applications for a problem involving the {$-\Delta_p(\cdot)$} operator
  and critical nonlinearity in {$\mathbb{R}^N$}}, Adv. Math. \textbf{265}
  (2014), 313--335.

\bibitem{DamascelliPacella1998}
L.~Damascelli and F.~Pacella, \emph{Monotonicity and symmetry of solutions of
  {$p$}-{L}aplace equations, {$1<p<2$}, via the moving plane method}, Ann.
  Scuola Norm. Sup. Pisa Cl. Sci. (4) \textbf{26} (1998), no.~4, 689--707.
  \MR{1648566}

\bibitem{DamascelliPacellaRamaswamy1999}
L.~Damascelli, F.~Pacella, and M.~Ramaswamy, \emph{Symmetry of ground states of
  {$p$}-{L}aplace equations via the moving plane method}, Arch. Ration. Mech.
  Anal. \textbf{148} (1999), no.~4, 291--308. \MR{1716666}

\bibitem{DamascelliSciunzi2004}
L.~Damascelli and B.~Sciunzi, \emph{Regularity, monotonicity and symmetry of
  positive solutions of {$m$}-{L}aplace equations}, J. Differential Equations
  \textbf{206} (2004), no.~2, 483--515.

\bibitem{FW1}
M.~Fall and T.~Weth, \emph{Nonexistence results for a class of fractional
  elliptic boundary value problems}, J. Funct. Anal. \textbf{263} (2012),
  no.~8, 2205--2227.

\bibitem{FarinaMontoroSciunzi2012}
A.~Farina, L.~Montoro, and B.~Sciunzi, \emph{Monotonicity and one-dimensional
  symmetry for solutions of {$-\Delta_pu=f(u)$} in half-spaces}, Calc. Var.
  Partial Differential Equations \textbf{43} (2012), no.~1--2, 123--145.

\bibitem{FarinaMontoroSciunzi2013}
\bysame, \emph{Monotonicity of solutions of quasilinear degenerate elliptic
  equations in half-spaces}, Math. Ann. \textbf{357} (2013), no.~3, 855--893.

\bibitem{GarainLindgren2024}
P.~Garain and E.~Lindgren, \emph{Higher {H}\"older regularity for the
  fractional {$p$}-{L}aplace equation in the subquadratic case}, Math. Ann.
  \textbf{390} (2024), no.~4, 5753--5792.

\bibitem{GNN}
B.~Gidas, W.-M. Ni, and L.~Nirenberg, \emph{Symmetry and related properties via
  the maximum principle}, Comm. Math. Phys. \textbf{68} (1979), 209--243.

\bibitem{GiovagnoliJesusSilvestre2025}
D.~Giovagnoli, D.~Jesus, and L.~Silvestre, \emph{{$C^{1+\alpha}$} regularity
  for fractional {$p$}-harmonic functions}, arXiv:2509.26565, 2025.

\bibitem{MR4788673}
A.~Iannizzotto and S.~Mosconi, \emph{Fine boundary regularity for the singular
  fractional {$p$}-{L}aplacian}, J. Differential Equations \textbf{412} (2024),
  322--379. \MR{4788673}

\bibitem{MR4985486}
\bysame, \emph{On a doubly sublinear fractional {$p$}-{L}aplacian equation},
  NoDEA Nonlinear Differential Equations Appl. \textbf{33} (2026), no.~1, Paper
  No. 14, 24 pp. \MR{4985486}

\bibitem{MR4604546}
A.~Iannizzotto, S.~Mosconi, and N.~S. Papageorgiou, \emph{On the logistic
  equation for the fractional {$p$}-{L}aplacian}, Math. Nachr. \textbf{296}
  (2023), no.~4, 1451--1468. \MR{4604546}

\bibitem{IMS2016}
A.~Iannizzotto, S.~J.~N. Mosconi, and M.~Squassina, \emph{Global {H}\"older
  regularity for the fractional {$p$}-{L}aplacian}, Rev. Mat. Iberoam.
  \textbf{32} (2016), no.~4, 1353--1392.

\bibitem{JW}
S.~Jarohs and T.~Weth, \emph{Symmetry via antisymmetric maximum principles in
  nonlocal problems of variable order}, Ann. Mat. Pura Appl. (4) \textbf{195}
  (2016), no.~1, 273--291.

\bibitem{JLX}
T.~Jin, Y.~Y. Li, and J.~Xiong, \emph{On a fractional {Nirenberg} problem,
  {Part I}: {Blow} up analysis and compactness of solutions}, J. Eur. Math.
  Soc. (JEMS) \textbf{16} (2014), no.~6, 1111--1171.

\bibitem{JX}
T.~Jin and J.~Xiong, \emph{A fractional {Yamabe} flow and some applications},
  J. Reine Angew. Math. \textbf{696} (2014), 187--223.

\bibitem{li2025direct}
C.~Li, M.~Xu, H.~Yang, and R.~Zhuo, \emph{The direct moving sphere for
  fractional {L}aplace equation}, J. Funct. Anal. \textbf{289} (2025), no.~8,
  Paper No. 111010, 29 pp. \MR{4897647}

\bibitem{Ly}
Y.~Y. Li, \emph{Remark on some conformally invariant integral equations: the
  method of moving spheres}, J. Eur. Math. Soc. (JEMS) \textbf{6} (2004),
  no.~2, 153--180. \MR{2055032}

\bibitem{liu2025nontrivial}
L.~Liu, \emph{Nontrivial nonnegative weak solutions to fractional
  {$p$}-{L}aplace inequalities}, arXiv:2501.11989, 2025.

\bibitem{MalekRajagopalRuzicka1995}
J.~M\'alek, K.~R. Rajagopal, and M.~R\r{u}\v{z}i\v{c}ka, \emph{Existence and
  regularity of solutions and the stability of the rest state for fluids with
  shear dependent viscosity}, Math. Models Methods Appl. Sci. \textbf{5}
  (1995), no.~6, 789--812.

\bibitem{MJ}
R.~Metzler and J.~Klafter, \emph{The random walk's guide to anomalous
  diffusion: A fractional dynamics approach}, Phys. Rep. \textbf{339} (2000),
  1--77.

\bibitem{MR2944369}
E.~Di Nezza, G.~Palatucci, and E.~Valdinoci, \emph{Hitchhiker's guide to the
  fractional {S}obolev spaces}, Bull. Sci. Math. \textbf{136} (2012), no.~5,
  521--573. \MR{2944369}

\bibitem{OlivaSciunziVaira2020}
F.~Oliva, B.~Sciunzi, and G.~Vaira, \emph{Radial symmetry for a quasilinear
  elliptic equation with a critical {Sobolev} growth and {Hardy} potential}, J.
  Math. Pures Appl. (9) \textbf{140} (2020), 89--109.

\bibitem{Sciunzi2005}
B.~Sciunzi, \emph{A weak maximum principle for the linearized operator of
  {$m$}-{L}aplace equations with applications to a nondegeneracy result}, Adv.
  Differential Equations \textbf{10} (2005), no.~2, 223--240.

\bibitem{Sciunzi2016}
\bysame, \emph{Classification of positive
  {$\mathcal{D}^{1,p}(\mathbb{R}^{N})$}-solutions to the critical
  {$p$}-{L}aplace equation in {$\mathbb{R}^{N}$}}, Adv. Math. \textbf{291}
  (2016), 12--23.

\bibitem{Serrin1971}
J.~Serrin, \emph{A symmetry problem in potential theory}, Arch. Ration. Mech.
  Anal. \textbf{43} (1971), 304--318.

\bibitem{SerrinZou1999}
J.~Serrin and H.~Zou, \emph{Symmetry of ground states of quasilinear elliptic
  equations}, Arch. Ration. Mech. Anal. \textbf{148} (1999), no.~4, 265--290.

\bibitem{Vazquez2020Evolution}
J.~L. V{\'a}zquez, \emph{The evolution fractional {$p$}-{L}aplacian equation in
  {$\mathbb R^N$}: Fundamental solution and asymptotic behaviour}, Nonlinear
  Anal. \textbf{199} (2020), Paper No. 112034, 32 pp.

\bibitem{Vazquez2021Sublinear}
\bysame, \emph{The fractional {$p$}-{L}aplacian evolution equation in {$\mathbb
  R^N$} in the sublinear case}, Calc. Var. Partial Differential Equations
  \textbf{60} (2021), no.~4, Paper No. 140, 59 pp.

\bibitem{WuChen2020}
L.~Wu and W.~Chen, \emph{The sliding methods for the fractional
  {$p$}-{L}aplacian}, Adv. Math. \textbf{361} (2020), Paper No. 106933, 26 pp.

\bibitem{WuYuZhang2021}
L.~Wu, M.~Yu, and B.~Zhang, \emph{Monotonicity results for the fractional
  {$p$}-{L}aplacian in unbounded domains}, Bull. Math. Sci. \textbf{11} (2021),
  no.~2, Paper No. 2150003.

\bibitem{Y}
H.~Yang, \emph{{Liouville-type theorems and radial symmetry for the Yamabe and
  Hardy-H\'enon equations involving higher-order fractional Laplacians}}, Sci.
  China Math. (2026), Published online, doi:10.1007/s11425-025-2642-8.

\bibitem{YeZhang}
D.~Ye and W.~Zhang, \emph{Comparison principles and symmetry for subquadratic
  fractional {$p$-Laplacian} equations}, arXiv:2609.02626 (2026).

\bibitem{ZCCY}
R.~Zhuo, W.~Chen, X.~W. Cui, and Z.~Yuan, \emph{Symmetry and non-existence of
  solutions for a nonlinear system involving the fractional {Laplacian}},
  Discrete Contin. Dyn. Syst. \textbf{36} (2016), 1125--1141.

\end{thebibliography}
\end{document}